\documentclass[12pt,reqno]{amsart}

\usepackage{amssymb,amsmath,graphicx,
amsfonts,euscript}
\usepackage{color}
\allowdisplaybreaks

\newtheorem{thm}{Theorem}[section]

\newtheorem{define}[thm]{Definition}

\newtheorem{lemma}[thm]{Lemma}

\def\na{\nabla}

\newcommand{\p}{\partial}

\newcommand{\beq}{\begin{equation}}
\newcommand{\eeq}{\end{equation}}
\newcommand{\ben}{\begin{eqnarray}}
\newcommand{\een}{\end{eqnarray}}
\newcommand{\beno}{\begin{eqnarray*}}
\newcommand{\eeno}{\end{eqnarray*}}

\numberwithin{equation}{section}

\begin{document}

\title[The magnetohydrodynamic equation]{Unique weak solutions of the non-resistive  magnetohydrodynamic equations with fractional dissipation}

\author[Jiu, Suo, Wu and Yu]{Quansen Jiu$^{1}$, Xiaoxiao Suo$^{2}$, Jiahong Wu$^{3}$ and Huan Yu$^{4}$}

\address{$^{1}$ School of Mathematics, Capital Normal University, Beijing 100037, P.R. China}

\email{jiuqs@mail.cnu.edu.cn}

\address{$^2$ School of Mathematics, Capital Normal University, Beijing 100037, P.R. China}

\email{xiaoxiao\_suo@163.com}

\address{$^3$ Department of Mathematics, Oklahoma State University, Stillwater, OK 74078, United States}

\email{jiahong.wu@okstate.edu}

\address{$^4$   School of Applied Science, Beijing Information Science and Technology University, Beijing, 100192, P.R.China}

\email{huanyu@bistu.edu.cn}

\vskip .2in
\begin{abstract}
This paper examines the uniqueness of weak solutions to the d-dimensional
magnetohydrodynamic (MHD) equations with the fractional dissipation $(-\Delta)^\alpha u$ and without the magnetic diffusion. Important progress has been made on the standard Laplacian dissipation case $\alpha=1$. This paper discovers that there are new phenomena with the case $\alpha<1$. The approach for $\alpha=1$ can not be directly extended to $\alpha<1$. We establish that, for $\alpha<1$, any initial data $(u_0, b_0)$ in the inhomogeneous Besov space $B^\sigma_{2,\infty}(\mathbb R^d)$ with $\sigma> 1+\frac{d}{2}-\alpha$ leads to a unique local solution. For the case $\alpha\ge 1$, $u_0$ in the homogeneous Besov space $\mathring B^{1+\frac{d}{2}-2\alpha}_{2,1}(\mathbb R^d)$ and $b_0$ in $ \mathring B^{1+\frac{d}{2}-\alpha}_{2,1}(\mathbb R^d)$ guarantees the existence and
uniqueness. These regularity requirements appear to be optimal.
\end{abstract}
\maketitle

\section{Introduction}

The MHD equations govern the motion of electrically conducting fluids such as plasmas, liquid metals, and electrolytes. They consist of a coupled
system of the Navier-Stokes equations of fluid dynamics
and Maxwell's equations of electromagnetism. Since their initial derivation by the Nobel Laureate H. Alfv\'{e}n in 1942 \cite{Alf}, the MHD equations have played pivotal roles in the study of many phenomena in geophysics,
astrophysics, cosmology and engineering (see, e.g., \cite{Bis,Davi}). Besides their wide physical applicability, the MHD equations are also of great interest in mathematics. As a coupled system, the MHD equations contain much richer structures than the Navier-Stokes
equations. They are not merely a combination of two parallel
Navier-Stokes type equations but an interactive and integrated system. Their distinctive features make analytic studies a great challenge but offer new opportunities.

\vskip .1in
Attention here is focused on the d-dimensional non-resistive MHD equation with fractional dissipation,
\beq \label{mhd}
\begin{cases}
\p_t u + u\cdot \na u + \nu (-\Delta)^\alpha u = -\na P + b\cdot\na b, \quad x \in \mathbb R^d, \,\, t>0, \\
\p_t b  + u\cdot\na b = b\cdot\na u, \quad x \in \mathbb R^d, \,\, t>0, \\
\na \cdot u =\na \cdot b =0, \quad x \in \mathbb R^d, \,\, t>0, \\
u(x,0) =u_0(x), \quad b(x,0) =b_0(x),\quad x \in \mathbb R^d,
\end{cases}
\eeq
where $u$, $P$ and $b$ represent the velocity, the pressure and the magnetic field, respectively, and $\nu>0$ is the kinematic viscosity and $\alpha>0$ is a parameter. The fractional Laplacian operator
$(-\Delta)^\alpha$ is defined via the Fourier transform,
$$
\widehat{(-\Delta)^\alpha f}(\xi) = |\xi|^{2\alpha}\, \widehat{f}(\xi),
$$
where
$$
\widehat{f}(\xi) = \frac{1}{(2\pi)^{d/2}} \int_{\mathbb R^d}
e^{-i x\cdot \xi} \, f(x) \,dx.
$$
When $\alpha =1$, (\ref{mhd}) reduces
to the standard MHD equation without magnetic diffusion, which
models electrically conducting  fluids that can be treated as perfect
conductors such as strongly collisional plasmas. When $\alpha>0$ is fractional, (\ref{mhd}) may be used to model nonlocal and long-range
diffusive interactions. Mathematically (\ref{mhd}) allows us to study a family of
equations simultaneously and provides a broad view on how the solutions are related to the sizes of $\alpha$.

\vskip .1in
One of the most fundamental issues on the MHD equations is the well-posedness problem. Mathematically rigorous foundational work has been laid by G. Duvaut and J. L. Lions in \cite{DuLi} and by M. Sermange and R. Temam in \cite{SeTe}. The MHD equations have recently gained renewed interests and there have been substantial developments on the well-posedness problem, especially when
the MHD equations involve only partial or fractional dissipation (see,
e.g., \cite{CaoReWu, CaoWu, CaoWuYuan, ChDai,DongLiWu1, DongLiWu0, FNZ,JNW,JiuZhao2,Wu4,
	Yam1,Yam2,Yam3,Yam4,WuYang,YuanZhao}). A summary
on some of the recent results can be found in a review paper \cite{WuMHD2018}.
(\ref{mhd}) with $\alpha \ge 1+\frac{d}{2}$
always has a unique global solution (see \cite{Wu2}). Yamazaki
was able to improve this result by weakening the dissipation
by a logarithm \cite{Yam4}. It remains an outstanding open problem whether
or not (\ref{mhd}) with $\alpha < 1+\frac{d}{2}$ can have
finite-time singular classical solutions. Even the global existence of Leray-Hopf weak solutions is not known due to the lack of suitable strong convergence in $b$.
In spite of the difficulties due to the lack of magnetic diffusion, significant progress has been made on the global well-posedness of solutions near background magnetic fields and many exciting results have been obtained (see,
e.g., \cite{BSS, CaiLei, HeXuYu,  HuX, HuLin, LinZhang1, PanZhouZhu, Ren, Ren2, Tan,WeiZ, WuWu, WuWuXu, TZhang}).

\vskip .1in
Another direction of research on the non-resistive MHD equation has resulted in
steady stream of progress. This direction has been seeking the weakest possible
functional setting for which one still has the uniqueness. The results currently available are for  (\ref{mhd}) with $\alpha=1$. Q. Jiu and D. Niu \cite{JiuNiu} proved the local well-poseness of (\ref{mhd}) with $\alpha=1$ in the Sobolev space $H^s$ with $s\ge 3$. Fefferman, McCormick, Robinson and Rodrigo were able to weaken the regularity assumption to
$(u_0, b_0)\in H^s$ with $s>\frac{d}{2}$ in \cite{FMRR1} and then to
$u_0 \in H^{s-1+\epsilon}$ and $b_0\in H^s$ with $s>\frac{d}{2}$ in \cite{FMRR2}. Chemin, McCormick, Robinson and Rodrigo \cite{CMRR} made further improvement by
assuming only $u_0\in B^{\frac{d}{2}-1}_{2,1}$ and
$b_0\in B^{\frac{d}{2}}_{2,1}$. Here $B^s_{p,q}$ denotes an inhomogeneous Besov space. They obtained the local existence for $d=2$ and $3$, and the uniqueness for $d=3$.
R. Wan \cite{Wan} obtained the uniqueness for $d=2$. J. Li, W. Tan and Z. Yin \cite{LTY} recently made an important progress by reducing the
functional setting to homogeneous Besov space $u_0\in \mathring B^{\frac{d}{p}-1}_{p,1}$ and $b_0\in \mathring B^{\frac{d}{p}}_{p,1}$ with $p\in [1, 2d]$. The definitions of the Besov spaces are provided in the appendix.

\vskip .1in
The aim of this paper is to establish the local existence and uniqueness of weak solutions with the minimal initial regularity assumption and for the largest possible range of $\alpha$'s. The case when $\alpha>1$ can be handled similarly as the case when $\alpha=1$. We can show that, for $\alpha>1$, any initial data $(u_0, b_0)$ with $u_0\in \mathring B^{\frac{d}{2}+ 1 -2\alpha}_{2,1} (\mathbb R^d)$ and $b_0\in \mathring B^{\frac{d}{2}+1-\alpha}_{2,1} (\mathbb R^d)$ leads to a unique local solution.

\vskip .1in
However, when $\alpha<1$, the situation is different and there are new phenomena. The approach for the case $\alpha=1$ can not be directly extended to  $\alpha<1$. We tested several seemingly natural classes of initial data:
\ben
&& (1) \,\,u_0 \in B_{2,1}^{\frac{d}{2}+ 1 -2\alpha} (\mathbb R^d) \quad \mbox{and} \quad b_0\in  B_{2,1}^{\frac{d}{2}} (\mathbb R^d);  \label{jjj8}\\
&& (2) \,\,u_0 \in B_{2,1}^{\frac{d}{2}+ 1 -2\alpha} (\mathbb R^d) \quad \mbox{and} \quad b_0\in  B_{2,1}^{\frac{d}{2} + 1 -\alpha} (\mathbb R^d); \label{jjj9}\\
&& (3)\,\,u_0 \in B_{2,1}^{\frac{d}{2}+ 1 -\alpha} (\mathbb R^d) \quad \mbox{and} \quad b_0\in  B_{2,1}^{\frac{d}{2} + 1 -\alpha} (\mathbb R^d),\label{jjj10}
\een
but it appears impossible to prove the local existence and uniqueness of weak solutions in these functional settings. Our investigation with these initial data leads to several discoveries. First, we realize that, in order to attain the uniqueness, the regularity level of the Besov space for $b_0$ has to have at least $\frac{d}{2}-\alpha +1$-derivative, which is more than $\frac{d}{2}$ for $\alpha<1$. Second, we discover that if the derivative of the Besov setting for $b_0$ exceeds $\frac{d}{2}$, then $u_0$ and $b_0$ should have the same Besov setting in order to establish the existence of solutions. Furthermore, one needs to combine the term of $b\cdot \na b$ in the equation of $u$ and $u\cdot\na b$ in the equation of $b$ in order to generate the cancellation. More technical explanations are given in
Section \ref{diss}. As a consequence of these findings, we choose the following
Besov spaces for $\alpha<1$,
$$
u_0 \in B_{2,\infty}^{\sigma} (\mathbb R^d), \quad b_0\in B_{2,\infty}^{\sigma} (\mathbb R^d), \quad \sigma >\frac{d}{2}+ 1 -\alpha.
$$
These functional settings appear to be optimal when $\alpha<1$. More technical evidence is provided in Section \ref{diss}.
Our precise result is stated in the following
theorem.

\begin{thm} \label{main}
Let $d\ge 2$ and consider (\ref{mhd}) with $0\le \alpha < 1+ \frac{d}{4}$.
Assume the initial data $(u_0, b_0)$ satisfies $\na\cdot u_0 = \na\cdot b_0=0$,
and is in the following Besov spaces
\ben
&& \mbox{for $\alpha \ge 1$}, \qquad u_0 \in \mathring B_{2,1}^{\frac{d}{2}+ 1 -2\alpha} (\mathbb R^d), \quad  b_0\in \mathring B_{2,1}^{\frac{d}{2}+1 -\alpha} (\mathbb R^d), \label{rr0}\\
&& \mbox{for $\alpha < 1$}, \qquad u_0 \in B_{2,\infty}^{\sigma} (\mathbb R^d), \quad b_0\in B_{2,\infty}^{\sigma} (\mathbb R^d), \quad \sigma >\frac{d}{2}+ 1 -\alpha. \label{rrr}
\een
Then there exist
$T>0$ and a unique local solution $(u, b)$ of (\ref{mhd}) satisfying, in  the case of $\alpha\ge 1$,
$$
u \in C([0, T]; \mathring B^{\frac{d}{2}+ 1 -2\alpha}_{2,1} (\mathbb R^d)) \cap L^1(0, T; \mathring B^{\frac{d}{2}+ 1}_{2,1}(\mathbb R^d)), \,
b\in C([0, T]; \mathring B_{2,1}^{\frac{d}{2}+1 -\alpha} (\mathbb R^d))
$$
and, in the case of $\alpha<1$,
$$
u \in C([0, T]; B^{\sigma}_{2,\infty} (\mathbb R^d)) \cap \widetilde L^2(0, T; B^{\alpha + \sigma}_{2,\infty}(\mathbb R^d)), \,
b\in C([0, T]; B_{2,\infty}^{\sigma} (\mathbb R^d)).
$$
\end{thm}

Theorem \ref{main} covers a full range of $\alpha \in [0, 1+\frac{d}{4})$
and includes $\alpha=1$ and $\alpha=0$ as two special cases. $\alpha < 1+\frac{d}{4}$ is imposed to satisfy a technical requirement in bounding the
high frequency interaction terms in the paraproduct decomposition. When $\alpha$ reaches this upper bound, the functional setting for $u_0$ is $\mathring B^{-1}_{2,1}$. When $\alpha \ge 1$, the initial data $(u_0, b_0)$ and the corresponding solution are in the homogeneous Besov spaces. For $\alpha<1$, the
functional setting are the inhomogeneous Besov spaces. We may not be able to reduce the assumption for $\alpha<1$ to the  corresponding homogeneous Besov spaces.

\vskip .1in
As aforementioned, the regularity assumptions imposed on  $(u_0, b_0)$ in
Theorem \ref{main} may be the minimal requirements we need for the existence
and uniqueness. We present a detailed explanation in Section \ref{diss}.
Roughly speaking,  when $\alpha\ge 1$, $u_0 \in \mathring B_{2,1}^{\frac{d}{2}+ 1 -2\alpha} (\mathbb R^d)$ in (\ref{rr0})  is necessary in order for the solution $u\in
L^1(0, T; \mathring B^{\frac{d}{2}+ 1}_{2,1}(\mathbb R^d))$, which is more or less the regularity level for the uniqueness. The regularity setting for $u_0$
leads to the corresponding choice for $b_0$, namely $b_0 \in\mathring  B_{2,1}^{\frac{d}{2}+1 -\alpha} (\mathbb R^d)$. In the case when $\alpha<1$, (\ref{rrr}) may be optimal due to our findings discovered in working with
three other initial Besov settings described above in (\ref{jjj8}), (\ref{jjj9}) and (\ref{jjj10}). Another hint comes from the uniqueness requirement for the ideal MHD equations. When $\alpha$ is zero or $\alpha>0$ is small, (\ref{rrr}) is the regularity class that guarantees the uniqueness of solutions to the ideal MHD equation.

\vskip .1in The statement of Theorem \ref{main} clearly indicates that the case
$\alpha\ge 1$ is handled differently from the case $\alpha<1$.  The existence part of
Theorem \ref{main} is proven through a successive approximation process. Naturally we divide the consideration into two cases: $\alpha\ge 1$ and $\alpha<1$. In the case when $\alpha\ge 1$, the successive approximation sequence $(u^{(n)}, b^{(n)})$ satisfies
\beq \label{sb}
\begin{cases}
u^{(1)} =\mathring S_2 u_0, \quad b^{(1)} =\mathring S_2 b_0, \\
\p_t u^{(n+1)} + \nu (-\Delta)^\alpha u^{(n+1)}
= \mathbb P (- u^{(n)}\cdot \na u^{(n+1)} + b^{(n)}\cdot \na b^{(n)}),\\
\p_t b^{(n+1)} = - u^{(n)} \cdot \na b^{(n+1)}
+ b^{(n)}\cdot \na u^{(n)}, \\
\na \cdot u^{(n+1)} =\na \cdot b^{(n+1)} =0,\\
u^{(n+1)} (x,0) =\mathring  S_{n+1} u_0, \quad b^{(n+1)} (x,0) = \mathring S_{n+1} b_0,
\end{cases}
\eeq
where $\mathbb P$ is the standard Leray projection and $\mathring S_j$ is the standard homogeneous low frequency cutoff operator (see the Appendix for its definition). The functional setting for $(u^{(n)}, b^{(n)})$ is given by
 \ben
 &&M= 2 \left( \|u_0\|_{\mathring B^{\frac{d}{2}+ 1 -2\alpha}_{2,1}}
 + \|b_0\|_{\mathring B_{2,1}^{\frac{d}{2}+ 1-\alpha}} \right) \notag \\
&&Y \equiv  \Big\{(u, b) \big|\,\, \|u\|_{\widetilde L^\infty(0,T; \mathring B^{\frac{d}{2}+ 1 -2\alpha}_{2,1})} \le M, \quad  \|b\|_{\widetilde L^\infty(0,T; \mathring B_{2,1}^{\frac{d}{2}+ 1-\alpha})}
\le M, \,\,\, \notag\\
&& \qquad\qquad  \|u\|_{L^1(0, T; \, \mathring B^{\frac{d}{2}+ 1}_{2,1})} \le \delta, \quad  \|u\|_{\widetilde L^2(0,T; \mathring B^{\frac{d}{2}+ 1-\alpha}_{2,1})} \le \delta\Big\}, \label{spaceY}
\een
where $T>0$ is chosen to be sufficiently small and $0< \delta <1$ is specified in
Section \ref{exi}. In the case when $\alpha<1$, $(u^{(n)}, b^{(n)})$ satisfies
\beq \label{sb1}
\begin{cases}
	u^{(1)} = S_2 u_0, \quad b^{(1)} = S_2 b_0, \\
	\p_t u^{(n+1)} + \nu (-\Delta)^\alpha u^{(n+1)}
	= \mathbb P (- u^{(n)}\cdot \na u^{(n+1)} + b^{(n)}\cdot \na b^{(n+1)}),\\
	\p_t b^{(n+1)} = - u^{(n)} \cdot \na b^{(n+1)}
	+ b^{(n)}\cdot \na u^{(n+1)}, \\
	\na \cdot u^{(n+1)} = \na \cdot b^{(n+1)} =0,\\
	u^{(n+1)} (x,0) = S_{n+1} u_0, \quad b^{(n+1)} (x,0) =  S_{n+1} b_0
\end{cases}
\eeq
and the corresponding functional setting is
 \ben
 && M= 2 \max\left\{\|(u_0, b_0)\|_{B^{\sigma}_{2,\infty}}, \frac1{\sqrt{C_0}}\|(u_0, b_0)\|_{B^{\sigma}_{2,\infty}}\right\},
 \notag \\
&& Y \equiv \Big\{(u, b) \big|\,\, \|(u, b)\|_{\widetilde L^\infty(0,T; B^{\sigma}_{2,\infty})} \le M,\,\, \|u\|_{\widetilde L^2(0, T; \, B^{\alpha +\sigma}_{2,\infty})} \le M \Big\}, \label{spaceY1}
\een
where $C_0>0$ is a pure constant as defined in (\ref{xj8}).
The main effort is devoted to showing that $(u^{(n)}, b^{(n)})$ actually converges
to a weak solution of (\ref{mhd}). The process of obtaining a subsequence of $(u^{(n)}, b^{(n)})$
that converges to a weak solution $(u, b)$ of  (\ref{mhd}) is divided into two main steps. The first step is to assert the uniform boundedness of $(u^{(n)}, b^{(n)})$
in $Y$ while the second step is to  extract a strongly convergent subsequence via the Aubin-Lions Lemma. The strong convergence then ensures that   the limit is indeed a weak solution of (\ref{mhd}). The uniform boundedness is shown via an iterative
process. We assume  $(u^{(n)}, b^{(n)}) \in Y$ and show $(u^{(n+1)}, b^{(n+1)}) \in Y$.

\vskip .1in
The technical approach to proving the uniform boundedness for the case
$\alpha\ge 1$ is different from that for the case when $\alpha<1$. For $\alpha< 1$,
we estimate $u^{(n+1)}$  and $b^{(n+1)}$ in $\widetilde L^\infty(0,T; B^{\sigma}_{2,\infty})$, and  $u^{(n+1)}$ in $\widetilde L^2(0, T; \, B^{\alpha +\sigma}_{2,\infty})$ simultaneously. The purpose is to make use of the cancellation
resulting from adding the equations for $\|\Delta_j u^{(n+1)}\|_{L^2}$ and $\|\Delta_j b^{(n+1)}\|_{L^2}$. The cancellation is in the sum
$$
\int_{\mathbb R^d} \Delta_j (b^{(n)}\cdot \na b^{(n+1)})\cdot \Delta_j u^{(n+1)}\,dx + \int_{\mathbb R^d} \Delta_j (b^{(n)}\cdot \na u^{(n+1)})\cdot \Delta_j b^{(n+1)}\,dx,
$$
whose paraproduct decomposition contains
$$
\int_{\mathbb R^d} \left(S_j b^{(n)}\cdot\na \Delta_j b^{(n+1)}\cdot \Delta_j u^{(n+1)} +S_j b^{(n)}\cdot\na \Delta_j u^{(n+1)}\cdot \Delta_jb^{(n+1)}\right)\,dx =0
$$
due to $\na\cdot b^{(n)}=0$. This appears to be the only approach that allows us to show the existence of
solutions in functional spaces with the order of the derivative exceeding $\frac{d}{2}$.
In the case when $\alpha\ge 1$, $b_0 \in \mathring B_{2,1}^{\frac{d}{2}+1 -\alpha} (\mathbb R^d)$ and the order of derivative is $\frac{d}{2}+1 -\alpha \le \frac{d}{2}$. The desired norms of $u^{(n+1)}$  and $b^{(n+1)}$ can be suitably estimated without the cancellation. In addition,  some upper bounds on products in Besov spaces are valid only for $\alpha\ge 1$ and break down when $\alpha<1$.
When $\alpha\ge 1$,
\beno
\|b^{(n)}\cdot \na b^{(n+1)}\|_{B^{\frac{d}{2} -2\alpha +1}_{2,1}(\mathbb R^d)} &\le& \|b^{(n)}\otimes b^{(n+1)}\|_{B^{\frac{d}{2}-2\alpha +2}_{2,1}(\mathbb R^d)}\\
&\le& C\, \|b^{(n)}\|_{B^{\frac{d}{2}-\alpha +1}_{2,1}(\mathbb R^d)}\, \|b^{(n+1)}\|_{B^{\frac{d}{2}-\alpha +1}_{2,1}(\mathbb R^d)}
\eeno
based on the following lemma (see, e.g., \cite[p.90]{BCD} or Lemma 2.6 in \cite{LTY}).
\begin{lemma} \label{ss}
	Let $1\le p\le \infty$, $s_1, s_2 \le \frac{d}{p}$ and $s_1 + s_2 >d\,\max\{0, \frac2p-1\}$. Then
	$$
	\|f\, g\|_{\dot B^{s_1 + s_2 -\frac{d}{p}}_{p,1}(\mathbb R^d)} \le C\, \|f\|_{\dot B^{s_1}_{p,1}(\mathbb R^d)} \, \|g\|_{\dot B^{s_2}_{p,1}(\mathbb R^d)}.
	$$
\end{lemma}
However, when $\alpha<1$, Lemma \ref{ss} breaks down since $s_1= \frac{d}{2} -\alpha +1$ and $s_2 =\frac{d}{2} -\alpha +1$ no longer satisfy the condition $s_1, s_2 \le \frac{d}{2}$.
This difficulty is overcome  by performing a detailed analysis on different frequencies of this product when $\alpha<1$.

\vskip .1in
The rest of this paper is divided into four sections and an appendix. Section \ref{exi} focuses on the proof of the existence part in Theorem \ref{main} for the case when $\alpha \ge 1$ while Section \ref{exi2} is devoted to the case when $\alpha < 1$. Section \ref{uni} presents the proof for the uniqueness part of Theorem \ref{main}. We again distinguish between the case when $\alpha \ge 1$ and the case when $\alpha<1$. Section \ref{diss} explains in detail why the regularity assumptions on the initial data in Theorem \ref{main} may be optimal. In particular,
we describe the difficulties when the regularity assumptions are reduced to those in (\ref{jjj8}), (\ref{jjj9}) and (\ref{jjj10}). The appendix provides the
definitions of Besov spaces and other closely related tools.

\vskip .3in
\section{Proof of the existence part in Theorem \ref{main} for $\alpha\ge 1$}
\label{exi}

This section proves the existence part of Theorem \ref{main} for the case  $\alpha\ge 1$. The approach is to construct a successive approximation sequence and
show that the limit of a subsequence actually solves
(\ref{mhd}) in the weak sense.

\vskip .1in
\begin{proof}[Proof for the existence part of Theorem \ref{main} in the case when $\alpha\ge 1$]
We consider a successive approximation sequence $\{(u^{(n)}, b^{(n)})\}$
satisfying (\ref{sb}). We define the functional
setting $Y$ as in (\ref{spaceY}) . Our goal is to show
that $\{(u^{(n)}, b^{(n)})\}$ has a subsequence that converges to a weak
solution of (\ref{mhd}). This process consists of three
main steps. The first step is to show that $(u^{(n)}, b^{(n)})$ is uniformly
bounded in $Y$. The second step is to extract a strongly convergent subsequence
via the Aubin-Lions Lemma while the last step is to show that the limit is
indeed a weak solution of (\ref{mhd}). Our main effort is devoted to showing
the uniform bound for $(u^{(n)}, b^{(n)})$ in $Y$. This is proven by induction.

\vskip .1in
We show inductively that $(u^{(n)}, b^{(n)})$ is bounded uniformly in $Y$.
Recall that $(u_0, b_0)$ is in the regularity class (\ref{rr0}).
According to (\ref{sb}),
$$
u^{(1)} =\mathring S_2 u_0, \qquad b^{(1)} = \mathring S_2 b_0.
$$
Clearly,
$$
\|u^{(1)}\|_{\widetilde L^\infty(0,T; \mathring B^{\frac{d}{2}+ 1 -2\alpha}_{2,1})}
\le M, \qquad  \|b^{(1)}\|_{\widetilde L^\infty(0,T; \mathring B_{2,1}^{\frac{d}{2}+1-\alpha})}
\le M.
$$
If $T>0$ is sufficiently small, then
\beno
&& \|u^{(1)}\|_{L^1(0,T; \mathring B^{\frac{d}{2}+ 1}_{2,1})} \le T \|\mathring S_2 u_0\|_{\mathring B^{\frac{d}{2}+ 1}_{2,1}} \le T\,C\, {\|u_0\|_{\mathring B_{2,1}^{\frac{d}{2}+1-2\alpha}}}
\le \delta,\\
&& \|u^{(1)}\|_{\widetilde L^2(0,T; \mathring B^{\frac{d}{2}+ 1-\alpha}_{2,1})} \le \sqrt{T}\,C\, {\|u_0\|_{\mathring B_{2,1}^{\frac{d}{2}+1-2\alpha}}}
\le \delta.
\eeno
Assuming that $(u^{(n)}, b^{(n)})$ obeys the bounds defined
in $Y$, namely
\beno
&& \|u^{(n)}\|_{\widetilde L^\infty(0,T; \mathring B^{\frac{d}{2}+ 1 -2\alpha}_{2,1})} \le M, \qquad
\|b^{(n)}\|_{\widetilde L^\infty(0,T; \mathring B_{2,1}^{\frac{d}{2}+1-\alpha})}
\le M, \\
&& \|u^{(n)}\|_{L^1(0,T; \mathring B^{\frac{d}{2}+ 1}_{2,1})} \le \delta, \qquad
\|u^{(n)}\|_{\widetilde L^2(0,T; \mathring B^{\frac{d}{2}+ 1-\alpha}_{2,1})} \le \delta,
\eeno
we prove that $(u^{(n+1)}, b^{(n+1)})$ obeys the same bounds for sufficiently small $T>0$ and suitably
selected $\delta>0$.  For the sake of clarity, the rest of this section is divided into five subsections.

\vskip .1in
\subsection{The estimate of
$u^{(n+1)}$ in $\widetilde L^\infty(0,T; \mathring B^{1+ \frac{d}{2} -2\alpha}_{2,1}(\mathbb R^d))$}

Let $j$ be an integer. Applying $\mathring \Delta_j$ (we shall just use $\Delta_j$ to simplify the notation) to the second equation in (\ref{sb})
and then dotting with $\Delta_j u^{(n+1)}$, we obtain
\beq \label{mm}
\frac12 \frac{d}{dt} \|\Delta_j u^{(n+1)}\|_{L^2}^2 + \nu \|\Lambda^\alpha
\Delta_j u^{(n+1)}\|_{L^2}^2 = A_1 + A_2,
\eeq
where
\beno
A_1 &=& - \int \Delta_j (u^{(n)}\cdot \na u^{(n+1)}) \cdot \Delta_j  u^{(n+1)}\, dx, \\
A_2 &=& \int \Delta_j (b^{(n)}\cdot \na b^{(n)})\cdot \Delta_j  u^{(n+1)}\,dx.
\eeno
We remark that the projection operator $\mathbb P$ has been eliminated due to the divergence-free condition $\na \cdot u^{(n+1)} =0$.  The dissipative part admits a lower bound
$$
\nu \|\Lambda^\alpha
\Delta_j u^{(n+1)}\|_{L^2}^2 \ge \,C_0\, 2^{2\alpha j}\, \|\Delta_j u^{(n+1)}\|_{L^2}^2,
$$
where $C_0>0$ is a constant.  According to Lemma \ref{para}, $A_1$ can be bounded by
\beno
|A_1| &\le& C\, \|\Delta_j  u^{(n+1)}\|_{L^2}^2 \, \sum_{m\le j-1}
2^{(1+\frac{d}{2})m} \|\Delta_m u^{(n)}\|_{L^2} \\
&& +\,C\, \|\Delta_j  u^{(n+1)}\|_{L^2}\,  \|\Delta_j  u^{(n)}\|_{L^2} \, \sum_{m\le j-1}
2^{(1+\frac{d}{2})m} \|\Delta_m u^{(n+1)}\|_{L^2} \\
&& +\,C\, \|\Delta_j  u^{(n+1)}\|_{L^2}\,2^j\,  \sum_{k\ge j-1} 2^{\frac{d}{2} k} \,{\|\Delta_k u^{(n)}\|_{L^2} \,\|\widetilde{\Delta}_k u^{(n+1)}\|_{L^2}}.
\eeno
Also by Lemma \ref{para}, $A_2$ is bounded by
\beno
|A_2| &\le& C\, \|\Delta_j  u^{(n+1)}\|_{L^2} \,2^j\, \|\Delta_j  b^{(n)}\|_{L^2} \, \sum_{m\le j-1}
2^{\frac{d}{2}\,m} \|\Delta_m b^{(n)}\|_{L^2} \\
&& +\,C\, \|\Delta_j  u^{(n+1)}\|_{L^2}\,  \|\Delta_j  b^{(n)}\|_{L^2} \, \sum_{m\le j-1}
2^{(1+\frac{d}{2})m} \|\Delta_m b^{(n)}\|_{L^2} \\
&& +\,C\, \|\Delta_j  u^{(n+1)}\|_{L^2}\,2^j\,  \sum_{k\ge j-1} 2^{\frac{d}{2} k} \,\|\Delta_k b^{(n)}\|_{L^2} \,
\|\widetilde{\Delta}_k b^{(n)}\|_{L^2}.
\eeno
Inserting the estimates above in (\ref{mm}) and eliminating $ \|\Delta_j  u^{(n+1)}\|_{L^2}$ from both sides of the inequality, we obtain
\beq \label{mmm1}
\frac{d}{dt} \|\Delta_j u^{(n+1)}\|_{L^2} + C_0\, 2^{2\alpha j} \|\Delta_j u^{(n+1)}\|_{L^2} \le J_1 +\cdots + J_6,
\eeq
where
\beno
J_1 &=&  C\, \|\Delta_j  u^{(n+1)}\|_{L^2}\, \sum_{m\le j-1}
2^{(1+\frac{d}{2})m} \|\Delta_m u^{(n)}\|_{L^2}, \\
J_2 &=& \,C\, \|\Delta_j  u^{(n)}\|_{L^2} \, {\sum_{m\le j-1}}
2^{(1+\frac{d}{2})m} \|\Delta_m u^{(n+1)}\|_{L^2} \\
J_3 &=&\,C\, 2^j\,  \sum_{k\ge j-1} 2^{\frac{d}{2} k} \,\|\Delta_k u^{(n)}\|_{L^2}\, \|\widetilde{\Delta}_k u^{(n+1)}\|_{L^2}, \\
J_4 &=&\,C\, 2^j\, \|\Delta_j  b^{(n)}\|_{L^2} \, \sum_{m\le j-1}
2^{\frac{d}{2}\,m} \|\Delta_m b^{(n)}\|_{L^2}, \\
J_5 &=&\,C\,  \|\Delta_j  b^{(n)}\|_{L^2} \, \sum_{m\le j-1}
2^{(1+\frac{d}{2})m} \|\Delta_m b^{(n)}\|_{L^2}, \\
J_6 &=&\,C\, 2^j\,  \sum_{k\ge j-1} 2^{\frac{d}{2} k} \,\|\Delta_k b^{(n)}\|_{L^2} \,\|\widetilde{\Delta}_k b^{(n)}\|_{L^2}.
\eeno
Integrating (\ref{mmm1}) in time yields
\ben
\|\Delta_j u^{(n+1)}(t)\|_{L^2} &\le& e^{-C_0\, 2^{2\alpha j} t}\, \|\Delta_j u_0^{(n+1)}\|_{L^2} \notag\\
&&  + \int_0^{t} e^{-C_0\, 2^{2\alpha j}({t} -\tau)} (J_1 +\cdots+ J_6)\,d\tau. \label{pp}
\een
Taking the $L^\infty(0, T)$ of (\ref{pp}), then multiplying by $2^{(1+ \frac{d}{2} -2\alpha) j}$ and summing over $j$, we have
\beno
&& \|u^{(n+1)} \|_{\widetilde L^\infty(0, T; B^{1+ \frac{d}{2} -2\alpha}_{2,1})}  \le  \|u^{(n+1)}_0\|_{B^{1+ \frac{d}{2} -2\alpha}_{2,1}}  \notag\\
&& \qquad \qquad  + \sum_{j} 2^{(1+ \frac{d}{2} -2\alpha) j} \left\|\int_0^{t}  e^{-C_0\, 2^{2\alpha j}({t} -\tau)} (J_1 +\cdots+ J_6)\,d\tau \right\|_{L^\infty(0, T)}.
\eeno
Applying Young's inequality to the convolution in time yields
\ben
&& \|u^{(n+1)} \|_{\widetilde L^\infty(0, T; B^{1+ \frac{d}{2} -2\alpha}_{2,1})}  \le  \|u^{(n+1)}_0\|_{B^{1+ \frac{d}{2} -2\alpha}_{2,1}}  \notag\\
&& \qquad \qquad \qquad  + \sum_{j} 2^{(1+ \frac{d}{2} -2\alpha) j} \left\|J_1 +\cdots+ J_6 \right\|_{L^1(0, T)}. \label{mmmm}
\een
The terms on the right-hand side can be estimated as follows. Recalling the definition of $J_1$ and using the inductive assumption on $u^{(n)}$, we have
\beno
&& \sum_{j} 2^{(1+ \frac{d}{2} -2\alpha) j} \int_0^T
\, J_1\,d \tau  \notag \\
&=& \,C\, \int_0^T \,\sum_{j} 2^{(1+ \frac{d}{2} -2\alpha) j} \|\Delta_j  u^{(n+1)}\|_{L^2}\, \sum_{m\le j-1}
2^{(1+\frac{d}{2})m} \|\Delta_m u^{(n)}\|_{L^2}\,d\tau   \notag\\
&\le & \,C\, \|u^{(n+1)}\|_{\widetilde L^\infty(0, T; \mathring B^{1+ \frac{d}{2} -2\alpha}_{2,1})} \, \|u^{(n)}\|_{L^1(0, T; \mathring B^{1+\frac{d}{2}}_{2,1})}   \notag\\
&\le& C\, \delta \, \|u^{(n+1)}\|_{\widetilde L^\infty(0, T; \mathring B^{1+ \frac{d}{2} -2\alpha}_{2,1})}.
\eeno
The term involving $J_2$ admits the same bound. In fact,
\beno
&& \sum_{j} 2^{(1+ \frac{d}{2} -2\alpha) j} \int_0^T
\, J_2\,d \tau \notag\\
&=& \,C\, \int_0^T\,\sum_{j} 2^{(1+ \frac{d}{2})j} \|\Delta_j  u^{(n)}\|_{L^2}\, {\sum_{m\le j-1}} 2^{2\alpha(m-j)}
2^{(1+\frac{d}{2}-2\alpha)m} \|\Delta_m u^{(n+1)}\|_{L^2}\,d\tau \notag\\
&\le& \,C\, \int_0^T\,\|u^{(n)}(\tau)\|_{\mathring B^{1+\frac{d}{2}}_{2,1}}\, \|u^{(n+1)}(\tau)\|_{\mathring B^{1+\frac{d}{2}-2\alpha}_{2,1}} \, d\tau
\notag\\
&\le& \,C\, \delta \, \|u^{(n+1)}\|_{\widetilde L^\infty(0, T; \mathring B^{1+ \frac{d}{2} -2\alpha}_{2,1})},
\eeno
where we have used the fact that $2\alpha(m-j) \le 0$. The term with $J_3$ is bounded by
\beno
&& \sum_{j} 2^{(1+ \frac{d}{2} -2\alpha) j} \int_0^T
\, J_3\,d \tau\notag \\
&=& \int_0^T \sum_{j} 2^{(1+ \frac{d}{2} -2\alpha)j} \,2^j\,  \sum_{k\ge j-1} 2^{\frac{d}{2} k} \,\|\widetilde{\Delta}_k u^{(n+1)}\|_{L^2} \,\|\Delta_k u^{(n)}\|_{L^2}\, d\tau \notag\\
&=& \,C\, \int_0^{T}\,\sum_{j}  \sum_{k\ge j-1}
2^{(2+ \frac{d}{2} -2\alpha)(j-1-k)}\, 2^{(1+ \frac{d}{2}) k} \|\Delta_k u^{(n)}\|_{L^2} \,\notag\\
&&\qquad \qquad \qquad \quad \times  2^{(1+ \frac{d}{2} -2\alpha) k} \|\widetilde{\Delta}_k u^{(n+1)}\|_{L^2}\,d\tau \notag\\
&\le& \,C\,  \int_0^T\, \|u^{(n)}(\tau)\|_{{\mathring B^{1+ \frac{d}{2}}_{2,1}}} \,
\|u^{(n+1)}(\tau)\|_{\mathring B^{1+ \frac{d}{2} -2\alpha}_{2,1}}\,d\tau \notag\\
&\le& \,C\, \delta \, \|u^{(n+1)}\|_{\widetilde L^\infty(0, T; \mathring B^{1+ \frac{d}{2} -2\alpha}_{2,1})},
\eeno
where we have used Young's inequality for series convolution and the fact $(2+ \frac{d}{2} -2\alpha)(j-1-k) < 0$. This is the place where we need
$\alpha< 1+ \frac{d}{4}$.
We now estimate the terms involving $J_4$ through $J_6$. The term with $J_4$ is bounded by,
\beno
&& \sum_{j} {2^{(1+ \frac{d}{2} -2\alpha)j}} \int_0^T
\, J_4\,d \tau \notag\\
&=& \sum_{j} \int_0^T  {2^{(1+ \frac{d}{2} -2\alpha) j}}\, 2^j\, \|\Delta_j  b^{(n)}\|_{L^2} \, \sum_{m\le j-1}
2^{\frac{d}{2}\,m} \|\Delta_m b^{(n)}\|_{L^2}\, d\tau \notag\\
&=& \int_0^T \sum_{j} {2^{(1+ \frac{d}{2} -\alpha) j}}\,\|\Delta_j  b^{(n)}\|_{L^2} \, \sum_{m\le j-1} 2^{(1-\alpha) (j-m)}\, 2^{(1+ \frac{d}{2} -\alpha) m}\|\Delta_m b^{(n)}\|_{L^2}\, d\tau \notag\\
&\le& \,C\,T\, \|b^{(n)}\|_{\widetilde L^\infty(0, T; \mathring B^{1+\frac{d}{2}-\alpha}_{2,1})}\,
 \|b^{(n)}\|_{\widetilde L^\infty(0, T; \mathring B^{1+\frac{d}{2}-\alpha}_{2,1})}\\
&\le& \,C\,T\, M^2,
\eeno
where we have used the fact that $\alpha\ge 1$ and $(1-\alpha) (j-m) \le 0$.
The terms with $J_5$ and $J_6$ are estimated similarly and they obey the same bound.
Inserting the bounds above in (\ref{mmmm}), we find
\ben
\|u^{(n+1)}\|_{\widetilde L^\infty(0, T; \mathring B^{1+ \frac{d}{2} -2\alpha}_{2,1})}   &\le& \|u^{(n+1)}_0\|_{\mathring B^{1+ \frac{d}{2} -2\alpha}_{2,1}}
+  \, C\, \delta \, \|u^{(n+1)}\|_{\widetilde L^\infty(0, T; \mathring B^{1+ \frac{d}{2} -2\alpha}_{2,1})}  \notag\\
&& + \,C\,T\, M^2. \label{gggg}
\een

\vskip .1in
\subsection{The estimate of $\|b^{(n+1)}\|_{\widetilde L^\infty(0, T;\mathring  B^{\frac{d}{2} -\alpha +1}_{2,1})}$}

We use the third equation of (\ref{sb}). Applying $\Delta_j$ to the third equation in (\ref{sb})
and then dotting with $\Delta_j b^{(n+1)}$, we obtain
\beq \label{mmg}
\frac12 \frac{d}{dt} \|\Delta_j b^{(n+1)}\|_{L^2}^2 \le  B_1 + B_2,
\eeq
where
\beno
B_1 &=& - \int \Delta_j (u^{(n)}\cdot \na b^{(n+1)}) \cdot \Delta_j  b^{(n+1)}\, dx, \\
B_2 &=& \int \Delta_j (b^{(n)}\cdot \na u^{(n)})\cdot \Delta_j  b^{(n+1)}\,dx.
\eeno
By Lemma \ref{para},
\beno
|B_1| &\le& C\, \|\Delta_j  b^{(n+1)}\|^2_{L^2}\,  \sum_{m\le j-1}
2^{(1+\frac{d}{2})m} \|\Delta_m u^{(n)}\|_{L^2} \\
&& +\,C\, \|\Delta_j  b^{(n+1)}\|_{L^2}\,  \|\Delta_j  u^{(n)}\|_{L^2} \, {\sum_{m\le j-1}}
2^{(1+\frac{d}{2})m} \|\Delta_m b^{(n+1)}\|_{L^2} \\
&& +\,C\, \|\Delta_j  b^{(n+1)}\|_{L^2}\,2^j\,  \sum_{k\ge j-1} 2^{\frac{d}{2} k} \,\|\widetilde{\Delta}_k b^{(n+1)}\|_{L^2}
\,\|\Delta_k u^{(n)}\|_{L^2}
\eeno
and
\beno
|B_2| &\le& C\, \|\Delta_j  b^{(n+1)}\|_{L^2} \,2^j\, \|\Delta_j  u^{(n)}\|_{L^2} \, \sum_{m\le j-1}
2^{\frac{d}{2}\,m} \|\Delta_m b^{(n)}\|_{L^2} \\
&& +\,C\, \|\Delta_j  b^{(n+1)}\|_{L^2}\,  \|\Delta_j  b^{(n)}\|_{L^2} \, \sum_{m\le j-1 }
2^{(1+\frac{d}{2})m} \|\Delta_m u^{(n)}\|_{L^2} \\
&& +\,C\, \|\Delta_j  b^{(n+1)}\|_{L^2}\,2^j\,  \sum_{k\ge j-1} 2^{\frac{d}{2} k} \,\|\Delta_k b^{(n)}\|_{L^2} \,
\|\widetilde{\Delta}_k u^{(n)}\|_{L^2}.
\eeno
Inserting the estimates above in (\ref{mmg}) and eliminating $ \|\Delta_j  b^{(n+1)}\|_{L^2}$ from both sides of the inequality, we obtain
\beq \label{xx}
\frac{d}{dt} \|\Delta_j b^{(n+1)}\|_{L^2} \le K_1 +\cdots + K_6,
\eeq
where
\beno
K_1 &=&  C\, \|\Delta_j  b^{(n+1)}\|_{L^2}\, \sum_{m\le j-1}
2^{(1+\frac{d}{2})m} \|\Delta_m u^{(n)}\|_{L^2}, \\
K_2 &=& \,C\, \|\Delta_j  u^{(n)}\|_{L^2} \, {\sum_{m\le j-1} }
2^{(1+\frac{d}{2})m} \|\Delta_m b^{(n+1)}\|_{L^2} \\
K_3 &=&\,C\, 2^j\,  \sum_{k\ge j-1} 2^{\frac{d}{2} k} \,\|\widetilde{\Delta}_k b^{(n+1)}\|_{L^2} \,\|\Delta_k u^{(n)}\|_{L^2}, \\
K_4 &=&\,C\, 2^j\, \|\Delta_j  u^{(n)}\|_{L^2} \, \sum_{m\le j-1}
2^{\frac{d}{2}\,m} \|\Delta_m b^{(n)}\|_{L^2}, \\
K_5 &=&\,C\,  \|\Delta_j  b^{(n)}\|_{L^2} \, \sum_{m\le j-1}
2^{(1+\frac{d}{2})m} \|\Delta_m u^{(n)}\|_{L^2}, \\
K_6 &=&\,C\, 2^j\,  \sum_{k\ge j-1} 2^{\frac{d}{2} k} \,\|\Delta_k b^{(n)}\|_{L^2} \,\|\widetilde{\Delta}_k u^{(n)}\|_{L^2}.
\eeno
Integrating (\ref{xx}) in time yields,
\ben
\|\Delta_j b^{(n+1)}(t)\|_{L^2}  &\le& \, \|\Delta_j b_0^{(n+1)}\|_{L^2}  + \int_0^t \,(K_1 +\cdots+ K_6)\,d\tau. \label{lady}
\een
Taking the $L^\infty(0, T)$ of (\ref{lady}), multiplying by $2^{(\frac{d}{2}-\alpha +1)j}$ and summing over $j$, we have
\ben
&& \|b^{(n+1)}\|_{\widetilde L^\infty(0, T; \mathring B^{\frac{d}{2}-\alpha +1}_{2,1})} \le \|b^{(n+1)}_0\|_{\mathring B^{\frac{d}{2}-\alpha +1}_{2,1}} \notag \\
&& \qquad\qquad  + \sum_{j} 2^{(\frac{d}{2}-\alpha +1)j}  \int_0^T  (K_1 +\cdots+ K_6)\,d\tau.\label{good}
\een
The terms on the right can be bounded similarly as those in the previous subsection. In fact,
\beno
\sum_{j} 2^{(\frac{d}{2}-\alpha +1)j}  \int_0^T  K_1 \,d\tau  &\le& \,C\, \|b^{(n+1)}\|_{\widetilde L^\infty(0, T; \mathring B^{\frac{d}{2}+1-\alpha}_{2,1})} \, \|u^{(n)}\|_{L^1(0, T; \mathring B^{1+\frac{d}{2}}_{2,1})}   \notag\\
&\le& C\, \delta \,\|b^{(n+1)}\|_{\widetilde L^\infty(0, T; \mathring B^{\frac{d}{2}+1-\alpha}_{2,1})}.
\eeno
Similarly,
\beno
&& \sum_{j} 2^{(\frac{d}{2}-\alpha +1)j}  \int_0^T  K_2 \,d\tau \le\, C\, \delta \,\|b^{(n+1)}\|_{\widetilde L^\infty(0, T; \mathring B^{\frac{d}{2}+1-\alpha}_{2,1})},\\
&& \sum_{j} 2^{(\frac{d}{2}-\alpha +1)j}  \int_0^T  K_3 \,d\tau \le\, C\, \delta \,\|b^{(n+1)}\|_{\widetilde L^\infty(0, T; \mathring B^{\frac{d}{2}+1-\alpha}_{2,1})}.
\eeno
The terms with $K_4$, $K_5$ and $K_6$ are bounded as follows.
\beno
&& \sum_{j} 2^{(\frac{d}{2}-\alpha +1)j}  \int_0^T  K_4 \,d\tau \\
&=& \,C\, \sum_{j} 2^{(\frac{d}{2}-\alpha +1)j}  \int_0^T  2^j\, \|\Delta_j  u^{(n)}\|_{L^2} \, \sum_{m\le j-1}
2^{\frac{d}{2}\,m} \|\Delta_m b^{(n)}\|_{L^2}\,d\tau\\
&=& \,C\, \int_0^T  \sum_{j} 2^{(\frac{d}{2} +1)j} \|\Delta_j  u^{(n)}\|_{L^2} \, \sum_{m\le j-1} 2^{(1-\alpha)(j-m)}\, 2^{(\frac{d}{2}+1-\alpha)\,m} \|\Delta_m b^{(n)}\|_{L^2}\,d\tau\\
 &\le& \,C\, \|u^{(n)}\|_{L^1(0, T; \mathring B^{1+\frac{d}{2}}_{2,1})}  \,\|b^{(n)}\|_{\widetilde L^\infty(0, T; \mathring B^{\frac{d}{2}+1-\alpha}_{2,1})}  \notag\\
&\le& C\, \delta \,M,
\eeno
where we have used the fact that $\alpha\ge 1$ and $(1-\alpha) (j-m) \le 0$. Similarly,
$$
\sum_{j} 2^{(\frac{d}{2}-\alpha +1)j}  \int_0^T  K_5 \,d\tau \le  C\, \delta \,M, \quad \sum_{j} 2^{(\frac{d}{2}-\alpha +1)j}  \int_0^T  K_6 \,d\tau \le  C\, \delta \,M.
$$
Inserting the estimates above in (\ref{good}), we find
\ben
\|b^{(n+1)}\|_{\widetilde L^\infty(0, T; \mathring B^{\frac{d}{2}-\alpha +1}_{2,1})} &\le& \|b^{(n+1)}_0\|_{\mathring B^{\frac{d}{2}-\alpha +1}_{2,1}}
+ C\, \delta \,\|b^{(n+1)}\|_{\widetilde L^\infty(0, T; \mathring B^{\frac{d}{2}+1-\alpha}_{2,1})} \notag\\
&& \,+ \,C\, \delta \,M. \label{bes}
\een

\vskip .1in
\subsection{The estimate of $\|u^{(n+1)}\|_{L^1\left(0, T;\,\mathring B^{1+\frac{d}{2}}_{2,1}\right)}$}

 We multiply (\ref{pp}) by $2^{(1+\frac{d}{2})j}$, sum over $j$ and integrate in time to obtain
\beno
\|u^{(n+1)}\|_{L^1\left(0, T;\mathring B^{1+\frac{d}{2}}_{2,1}\right)} &\le& \int_0^T \sum_{j} 2^{(1+\frac{d}{2}) j}\, e^{-C_0\, 2^{2\alpha j} t}\, \|\Delta_j u_0^{(n+1)}\|_{L^2}\,dt
\\
&&  + \int_0^T \sum_{j} 2^{(1+\frac{d}{2}) j}\, \int_0^s e^{-C_0\, 2^{2\alpha j}(s-\tau)} (J_1 +\cdots+ J_6)\,d\tau\, ds.
\eeno
We estimate the terms on the right and start with the first term.
\beno
&&\int_0^T \sum_{j} 2^{(1+\frac{d}{2}) j}\, e^{-C_0\, 2^{2\alpha j} t}\, \|\Delta_j u_0^{(n+1)}\|_{L^2}\,dt  \notag\\
&=& C\, \sum_{j} 2^{(1+\frac{d}{2}-2\alpha) j} \left(1- e^{-C_0\, 2^{2\alpha j} T}\right)\, \|\Delta_j u_0^{(n+1)}\|_{L^2}.
\eeno
Since $u_0 \in \mathring B^{1+\frac{d}{2}-2\alpha}_{2,1}$, it follows from the Dominated Convergence Theorem that
$$
\lim_{T\to 0}\,  \sum_{j} 2^{(1+\frac{d}{2}-2\alpha) j} \left(1- e^{-C_0\, 2^{2\alpha j} T}\right)\, \|\Delta_j u_0^{(n+1)}\|_{L^2} = 0.
$$
Therefore, we can choose $T$ sufficiently small such that
$$
\int_0^T \sum_{j} 2^{(1+\frac{d}{2}) j}\, e^{-C_0\, 2^{2\alpha j} t}\, \|\Delta_j u_0^{(n+1)}\|_{L^2}\,dt \le \frac{\delta}{4}.
$$
Applying Young's inequality for the time convolution, we have
\beno
&& \int_0^T \sum_{j} 2^{(1+\frac{d}{2}) j}\, \int_0^s e^{-C_0\, 2^{2\alpha j}(s-\tau)} J_1\,d \tau\,ds \\
&=& \,C\, \int_0^T  \sum_{j} 2^{(1+\frac{d}{2}) j}\, \int_0^s e^{-C_0\, 2^{2\alpha j}(s-\tau)} \|\Delta_j  u^{(n+1)}(\tau)\|_{L^2}\, \\
&& \qquad \qquad \qquad \qquad \qquad \times \sum_{m\le j-1}
2^{(1+\frac{d}{2})m} \|\Delta_m u^{(n)}(\tau)\|_{L^2}\,d\tau \,ds \\
&\le& \,C\, \int_0^T \sum_{j} 2^{(1+\frac{d}{2}) j}\,\|\Delta_j  u^{(n+1)}\|_{L^2}\, \sum_{m\le j-1}
2^{(1+\frac{d}{2})m} \|\Delta_m u^{(n)}\|_{L^2}\, d\tau \\
&& \qquad\qquad \qquad  \times
\int_{0}^T
e^{-C_0\, 2^{2\alpha j}s} ds\,d\tau\\
&\le& \,C\, \int_0^T \sum_{j} 2^{(1+\frac{d}{2}-\alpha) j}\,\|\Delta_j  u^{(n+1)}(\tau)\|_{L^2}\, \\
&& \qquad\qquad   \times \sum_{m\le j-1} 2^{(m-j)\alpha}\,
2^{(1+\frac{d}{2}-\alpha)m} \|\Delta_m u^{(n)}(\tau)\|_{L^2} \,d\tau \\
&\le& \,C\,\|u^{(n)}\|_{\widetilde L^2(0, T;\mathring B_{2,1}^{1+\frac{d}{2}-\alpha})}\, \|u^{(n+1)}\|_{\widetilde L^2(0, T;\mathring B_{2,1}^{1+\frac{d}{2}-\alpha})} \\
&\le& \,C\, \delta \, \|u^{(n+1)}\|_{\widetilde L^2(0, T; \mathring B_{2,1}^{1+\frac{d}{2}-\alpha})}.
\eeno
The terms with $J_2$ and $J_3$ can be estimated similarly and they share the same upper bound with the term involving $J_1$,
\beno
&& \int_0^T \sum_{j} 2^{(1+\frac{d}{2}) j}\, \int_0^s e^{-C_0\, 2^{2\alpha j}(s-\tau)} J_2\,d \tau\,ds  \le \,C\, \delta \, \|u^{(n+1)}\|_{\widetilde L^2(0, T; \mathring B_{2,1}^{1+\frac{d}{2}-\alpha})},  \\
&& \int_0^T \sum_{j} 2^{(1+\frac{d}{2}) j}\, \int_0^s e^{-C_0\, 2^{2\alpha j}(s-\tau)} J_3\,d \tau\,ds  \le \,C\, \delta \, \|u^{(n+1)}\|_{\widetilde L^2(0, T; \mathring B_{2,1}^{1+\frac{d}{2}-\alpha})}.
\eeno
We now examine the terms involving $J_4$ through $J_6$. Again by Young's inequality,
\beno
&& \int_0^T \sum_{j} 2^{(1+\frac{d}{2}) j}\, \int_0^s e^{-C_0\, 2^{2\alpha j}(s-\tau)} J_4\,d \tau\,ds \\
&=& \,C\, \int_0^T  \sum_{j} 2^{(1+\frac{d}{2}) j}\, \int_0^s e^{-C_0\, 2^{2\alpha j}(s-\tau)} 2^j\, \|\Delta_j  b^{(n)}(\tau)\|_{L^2} \,\\
&& \qquad\qquad\qquad\qquad \quad \times  \sum_{m\le j-1}
2^{\frac{d}{2}\,m} \|\Delta_m b^{(n)}(\tau)\|_{L^2} \,d \tau\,ds \\
&\le& \,C\,\int_0^T  \sum_{j} 2^{(\frac{d}{2} + 2-2\alpha) j}\,\|\Delta_j  b^{(n)}(\tau)\|_{L^2} \, \sum_{m\le j-1}
2^{\frac{d}{2}\,m} \|\Delta_m b^{(n)}(\tau)\|_{L^2} \,d\tau \\
&\le& \,C\,\int_0^T  \sum_{j} 2^{(\frac{d}{2} + 1-\alpha) j}\,\|\Delta_j  b^{(n)}(\tau)\|_{L^2} \, \sum_{m\le j-1} 2^{(1-\alpha)(j-m)}\,
2^{(\frac{d}{2}+1-\alpha)\,m} \|\Delta_m b^{(n)}(\tau)\|_{L^2} \,d\tau \\
&\le& \,C\, \int_0^T  \|b^{(n)}(\tau)\|^2_{\mathring B_{2,1}^{1+\frac{d}{2}-\alpha}}\,
d\tau\\
&\le & \,C\,T\, \|b^{(n)}\|^2_{\widetilde L^\infty(0, T; \mathring B_{2,1}^{1+\frac{d}{2}-\alpha})} \le \, C\, T\,M^2,
\eeno
where we have used the fact that $\alpha\ge 1$ and $(1-\alpha) (j-m)\le 0$ again.  The other two terms involving  $J_5$ and $J_6$ obey the same bound,
\beno
&& \int_0^T \sum_{j} 2^{(1+\frac{d}{2}) j}\, \int_0^s e^{-C_0\, 2^{2\alpha j}(s-\tau)} J_5\,d \tau\,ds  \le \, C\, T\, M^2,\\
&& \int_0^T \sum_{j} 2^{(1+\frac{d}{2}) j}\, \int_0^s e^{-C_0\, 2^{2\alpha j}(s-\tau)} J_6\,d \tau\,ds  \le \,\, C\, T\, M^2.
\eeno
Here we have used $\alpha<1+\frac{d}{4}$ in the estimate of $J_6$. Collecting the estimates above leads to
\ben
\|u^{(n+1)}\|_{L^1\left(0, T; \mathring B^{1+\frac{d}{2}}_{2,1}\right)} \le \frac{\delta}{4}  + C\,\delta  \, \|u^{(n+1)}\|_{\widetilde L^2(0, T; \mathring B_{2,1}^{1+\frac{d}{2}-\alpha})} + \, C\, T\,M^2.\label{ppo}
\een

\vskip .1in
\subsection{The bound for  $\|u^{(n+1)}\|_{\widetilde L^2(0, T; \mathring B_{2,1}^{1+\frac{d}{2}-\alpha})}$}

We multiply (\ref{pp}) by $2^{(1+\frac{d}{2}-\alpha)j}$, take the $L^2(0, T)$-norm
and sum over $j$  to obtain
\ben
&& \|u^{(n+1)}\|_{\widetilde L^2(0, T; \mathring B_{2,1}^{1+\frac{d}{2}-\alpha})} \le \sum_{j} 2^{(1+\frac{d}{2}-\alpha) j}\, \left\|e^{-C_0\, 2^{2\alpha j} t}\, \|\Delta_j u_0^{(n+1)}\|_{L^2} \,\right\|_{L^2(0,T)}
\notag\\
&& \qquad\quad + \sum_{j} 2^{(1+\frac{d}{2}-\alpha) j}\, \left\|\int_0^s e^{-C_0\, 2^{2\alpha j}(s-\tau)} (J_1 +\cdots+ J_6)\,d\tau\,\right\|_{L^2(0,T)}.
\label{xxy}
\een
The first term on the right is bound by
\beno
&& \sum_{j} 2^{(1+\frac{d}{2}-\alpha) j}\, \left\|e^{-C_0\, 2^{2\alpha j} t}\, \|\Delta_j u_0^{(n+1)}\|_{L^2} \,\right\|_{L^2(0,T)}\notag\\
&=& C\, \sum_{j} 2^{(1+\frac{d}{2}-2\alpha) j} \left(1- e^{-2 C_0\, 2^{2\alpha j} T}\right)^{\frac12}\, \|\Delta_j u_0^{(n+1)}\|_{L^2}.
\eeno
Since $u_0 \in \mathring B^{1+\frac{d}{2}-2\alpha}_{2,1}$, it follows from the Dominated Convergence Theorem that
$$
\lim_{T\to 0}\,  \sum_{j} 2^{(1+\frac{d}{2}-2\alpha) j} \left(1- e^{-2C_0\, 2^{2\alpha j} T}\right)^{\frac12}\, \|\Delta_j u_0^{(n+1)}\|_{L^2} = 0.
$$
Therefore we can choose $T>0$ sufficiently small such that
$$
\sum_{j} 2^{(1+\frac{d}{2}-\alpha) j}\, \left\|e^{-C_0\, 2^{2\alpha j} t}\, \|\Delta_j u_0^{(n+1)}\|_{L^2} \,\right\|_{L^2(0,T)} \le \frac{\delta}{4}.
$$
The other six terms on the right of (\ref{xxy}) are estimated as follows.
Applying Young's inequality for the time convolution, we have
\beno
&& \sum_{j} 2^{(1+\frac{d}{2}-\alpha) j}\, \left\|\int_0^s e^{-C_0\, 2^{2\alpha j}(s-\tau)} J_1\,d\tau\,\right\|_{L^2(0,T)} \\
&=& \,C\, \sum_{j} 2^{(1+\frac{d}{2}-\alpha) j}\, \Big\|\int_0^s e^{-C_0\, 2^{2\alpha j}(s-\tau)} \|\Delta_j  u^{(n+1)}(\tau)\|_{L^2}\, \\
&& \qquad \qquad \qquad \qquad \qquad \times \sum_{m\le j-1}
2^{(1+\frac{d}{2})m} \|\Delta_m u^{(n)}(\tau)\|_{L^2}\,d\tau\Big\|_{L^2(0,T)} \\
&\le& \,C\,\sum_{j} 2^{(1+\frac{d}{2}-\alpha) j}\,\|e^{-C_0\, 2^{2\alpha j}s}\|_{L^2(0,T)} \\
&&\qquad\qquad\qquad \times\,\Big\|\|\Delta_j  u^{(n+1)}(\tau)\|_{L^2} \sum_{m\le j-1}
2^{(1+\frac{d}{2})m} \|\Delta_m u^{(n)}(\tau)\|_{L^2}\Big\|_{L^1(0, T)}\\
&\le& \,C\, \int_0^T \sum_{j} 2^{(1+\frac{d}{2}-2\alpha) j}\,\|\Delta_j  u^{(n+1)}\|_{L^2}\, \sum_{m\le j-1}
2^{(1+\frac{d}{2})m} \|\Delta_m u^{(n)}\|_{L^2}\, d\tau \\
&\le& \,C\, \int_0^T \sum_{j} 2^{(1+\frac{d}{2}-\alpha) j}\,\|\Delta_j  u^{(n+1)}(\tau)\|_{L^2}\, \\
&& \qquad\qquad   \times \sum_{m\le j-1} 2^{(m-j)\alpha}\,
2^{(1+\frac{d}{2}-\alpha)m} \|\Delta_m u^{(n)}(\tau)\|_{L^2} \,d\tau \\
&\le& \,C\,\|u^{(n)}\|_{\widetilde L^2(0, T;\mathring B_{2,1}^{1+\frac{d}{2}-\alpha})}\, \|u^{(n+1)}\|_{\widetilde L^2(0, T;\mathring B_{2,1}^{1+\frac{d}{2}-\alpha})} \\
&\le& \,C\, \delta \, \|u^{(n+1)}\|_{\widetilde L^2(0, T; \mathring B_{2,1}^{1+\frac{d}{2}-\alpha})}.
\eeno
The terms with $J_2$ and $J_3$ share the same upper bound,
\beno
&&\sum_{j} 2^{(1+\frac{d}{2}-\alpha) j}\, \left\|\int_0^s e^{-C_0\, 2^{2\alpha j}(s-\tau)} J_2\,d\tau\,\right\|_{L^2(0,T)}  \le \,C\, \delta \, \|u^{(n+1)}\|_{\widetilde L^2(0, T; \mathring B_{2,1}^{1+\frac{d}{2}-\alpha})},  \\
&& \sum_{j} 2^{(1+\frac{d}{2}-\alpha) j}\, \left\|\int_0^s e^{-C_0\, 2^{2\alpha j}(s-\tau)} J_3\,d\tau\,\right\|_{L^2(0,T)} \le \,C\, \delta \, \|u^{(n+1)}\|_{\widetilde L^2(0, T; \mathring B_{2,1}^{1+\frac{d}{2}-\alpha})}.
\eeno
The estimate of the term with $J_4$ is similar. Again by Young's inequality,
\beno
&& \sum_{j} 2^{(1+\frac{d}{2}-\alpha) j}\, \left\|\int_0^s e^{-C_0\, 2^{2\alpha j}(s-\tau)} J_4\,d\tau\,\right\|_{L^2(0,T)}  \\
&=& \,C\, \sum_{j} 2^{(1+\frac{d}{2}-\alpha) j}\, \Big\|\int_0^s e^{-C_0\, 2^{2\alpha j}(s-\tau)} 2^j\, \|\Delta_j  b^{(n)}(\tau)\|_{L^2} \,\\
&& \qquad\qquad\qquad\qquad \quad \times  \sum_{m\le j-1}
2^{\frac{d}{2}\,m} \|\Delta_m b^{(n)}(\tau)\|_{L^2} \,d \tau\,ds\Big\|_{L^2(0,T)} \\
&\le& \,C\, \sum_{j} 2^{(\frac{d}{2} + 2-2\alpha) j}\,\Big\|\|\Delta_j  b^{(n)}(\tau)\|_{L^2} \, \sum_{m\le j-1}
2^{\frac{d}{2}\,m} \|\Delta_m b^{(n)}(\tau)\|_{L^2}\Big\|_{L^1(0,T)} \\
&\le& \,C\,\int_0^T  \sum_{j} 2^{(\frac{d}{2} + 1-\alpha) j}\,\|\Delta_j  b^{(n)}(\tau)\|_{L^2} \,\\
&&\qquad\qquad\times \sum_{m\le j-1} 2^{(1-\alpha)(j-m)}\,
2^{(\frac{d}{2}+1-\alpha)\,m} \|\Delta_m b^{(n)}(\tau)\|_{L^2} \,d\tau \\
&\le& \,C\, \int_0^T  \|b^{(n)}(\tau)\|^2_{\mathring B_{2,1}^{1+\frac{d}{2}-\alpha}}\,
d\tau\\
&\le& \,C\,T\, \|b^{(n)}\|^2_{\widetilde L^\infty(0, T; \mathring B_{2,1}^{1+\frac{d}{2}-\alpha})} \le \, C\, T\,M^2.
\eeno
The other two terms involving  $J_5$ and $J_6$ obey the same bound,
\beno
&& \sum_{j} 2^{(1+\frac{d}{2}-\alpha) j}\, \left\|\int_0^s e^{-C_0\, 2^{2\alpha j}(s-\tau)} J_5\,d\tau\,\right\|_{L^2(0,T)}  \le \, C\, T\, M^2,\\
&&\sum_{j} 2^{(1+\frac{d}{2}-\alpha) j}\, \left\|\int_0^s e^{-C_0\, 2^{2\alpha j}(s-\tau)} J_6\,d\tau\,\right\|_{L^2(0,T)}   \le \,\, C\, T\, M^2.
\eeno
Collecting the estimates above leads to
\ben
\|u^{(n+1)}\|_{\widetilde L^2(0, T; \mathring B_{2,1}^{1+\frac{d}{2}-\alpha})} \le \frac{\delta}{4}  + C\,\delta  \, \|u^{(n+1)}\|_{\widetilde L^2(0, T; \mathring B_{2,1}^{1+\frac{d}{2}-\alpha})} + \, C\, T\,M^2.\label{dj8}
\een

\vskip .1in
\subsection{Completion of the proof for the existence part in the case when $\alpha\ge 1$}

The bounds in (\ref{gggg}), (\ref{bes}), (\ref{ppo}) and (\ref{dj8}) allow us to
conclude that, if we choose $T>0$ sufficiently small and $\delta>0$ suitably, then
\beno
&& \|u^{(n+1)}\|_{\widetilde L^\infty(0,T; \mathring B^{\frac{d}{2}+ 1 -2\alpha}_{2,1})} \le M, \qquad
\|b^{(n+1)}\|_{\widetilde L^\infty(0,T; \mathring B_{2,1}^{\frac{d}{2}+1-\alpha})}
\le M, \\
&& \|u^{(n+1)}\|_{L^1(0,T; \mathring B^{\frac{d}{2}+ 1}_{2,1})} \le \delta, \qquad
\|u^{(n+1)}\|_{\widetilde L^2(0,T; \mathring B^{\frac{d}{2}+ 1-\alpha}_{2,1})} \le \delta.
\eeno
In fact, if $T$ and $\delta$ in (\ref{gggg}) satisfy
$$
C\delta \le \frac14, \qquad C\,  T\,M \le \frac14,
$$
then (\ref{gggg}) implies
$$
\|u^{(n+1)}\|_{\widetilde L^\infty(0,T; \mathring B^{\frac{d}{2}+ 1 -2\alpha}_{2,1})} \le \frac12 M + \frac14 \|u^{(n+1)}\|_{\widetilde L^\infty(0,T; \mathring B^{\frac{d}{2}+ 1 -2\alpha}_{2,1})} + \frac14 M
$$
or
$$
\|u^{(n+1)}\|_{\widetilde L^\infty(0,T; \mathring B^{\frac{d}{2}+ 1 -2\alpha}_{2,1})}  \le M.
$$
Similarly, if $C\delta \le \frac14$ in (\ref{bes}), then (\ref{bes}) states
$$
\|b^{(n+1)}\|_{\widetilde L^\infty(0,T; \mathring B_{2,1}^{\frac{d}{2}+1-\alpha})}
\le M.
$$
According to  (\ref{dj8}), if we choose $C\delta \le \frac14$ and $C\, T\,M^2 \le \frac12 \delta$, then
$$
\|u^{(n+1)}\|_{\widetilde L^2(0,T; \mathring B^{\frac{d}{2}+ 1-\alpha}_{2,1})} \le \delta
$$
and consequently, if $C\delta \le \frac14$ and $C\, T\,M^2\le \frac12 \delta$ in (\ref{ppo}), then
$$
\|u^{(n+1)}\|_{\widetilde L^1(0,T; \mathring B^{\frac{d}{2}+ 1}_{2,1})} \le \delta.
$$
These uniform bounds allow us to extract a weakly convergent subsequence. That is,
there is $(u, b)\in Y$ such that a subsequence of $(u^{(n)}, b^{(n)})$ (still denoted by $(u^{(n)}, b^{(n)})$) satisfies
\beno
&& u^{(n)}  \overset{\ast}{\rightharpoonup} u \quad \mbox{in}\quad \widetilde L^\infty(0,T; \mathring B^{\frac{d}{2}+ 1 -2\alpha}_{2,1}), \\
&& b^{(n)}  \overset{\ast}{\rightharpoonup} b \quad \mbox{in}\quad \widetilde L^\infty(0,T; \mathring B_{2,1}^{\frac{d}{2}+1-\alpha}).
\eeno
In order to show that $(u, b)$ is a weak solution of (\ref{mhd}), we need to further
extract a subsequence which converges strongly to $(u, b)$. This is done via the Aubin-Lions Lemma. We can show by making use of the equations in (\ref{sb}) that $(\p_t u^{(n)},\p_t b^{(n)})$ is {uniformly} bounded in
\beno
&& \p_t u^{(n)} \in  L^1(0, T; \mathring B^{\frac{d}{2}-2\alpha +1}_{2,1}) \cap \widetilde L^2(0, T; \mathring B^{\frac{d}{2}+1-3\alpha}_{2,1}), \\
&& \p_t b^{(n)}\in L^2(0, T; \mathring B^{\frac{d}{2}+1 -2\alpha}_{2,1})
\eeno
Since we are in the case of the whole space $\mathbb R^d$, we need to combine Cantor's diagonal process with the Aubin-Lions Lemma to show that a subsequence of the weakly convergent subsequence, still denoted by $(u^{(n)}, b^{(n)})$, has the following strongly convergent property,
\beno
(u^{(n)}, b^{(n)}) \to (u, b) \qquad \mbox{in} \quad L^2(0, T; \mathring B^\gamma_{2,1}(Q)),
\eeno
where $\frac{d}{2}+1 - 2\alpha \le \gamma < \frac{d}{2} + 1 -\alpha$ and $Q\subset \mathbb R^d$ is any compact subset. This strong convergence property would allow us to show that $(u, b)$ is indeed a weak solution of (\ref{mhd}). This process is routine and we omit the details. This completes the proof for the existence part of Theorem \ref{main}
in the case when $\alpha\ge 1$.
\end{proof}

\vskip .3in
\section{Proof of the existence part in Theorem \ref{main} with $\alpha< 1$}
\label{exi2}

This section proves the existence part of Theorem \ref{main} for the case when $\alpha< 1$. The idea is still to construct a successive approximation sequence and
show that the limit of a subsequence actually solves
(\ref{mhd}) in the weak sense. Some of the technical approaches here are different from those for $\alpha\ge 1$.

\vskip .1in
\begin{proof}[Proof for the existence part of Theorem \ref{main} in the case when $\alpha< 1$]
	We consider a successive approximation sequence $\{(u^{(n)}, b^{(n)})\}$
	satisfying (\ref{sb1}). We define the functional
	setting $Y$ as in (\ref{spaceY1}). Our goal is to show
	that $\{(u^{(n)}, b^{(n)})\}$ has a subsequence that converges to the weak
	solution of (\ref{mhd}). This process consists of three
	main steps. The first step is to show that $(u^{(n)}, b^{(n)})$ is uniformly
	bounded in $Y$. The second step is to extract a strongly convergent subsequence
	via the Aubin-Lions Lemma while the last step is to show that the limit is
	indeed a weak solution of (\ref{mhd}). Our main effort is devoted to showing
	the uniform bound for $(u^{(n)}, b^{(n)})$ in $Y$. This is proven by induction.
    We start with the basis step. Recall that $(u_0, b_0)$ is in the regularity class (\ref{rrr}).
    According to (\ref{sb1}),
    $$
    u^{(1)} =S_2 u_0, \qquad b^{(1)} = S_2 b_0.
    $$
    Clearly,
    $$
    \|u^{(1)}\|_{\widetilde L^\infty(0,T;  B^{\sigma}_{2,\infty})}
    \le M, \qquad  \|b^{(1)}\|_{\widetilde L^\infty(0,T;  B_{2,\infty}^{\sigma})}
    \le M.
    $$
    If $T>0$ is sufficiently small, then
    \beno
    && \|u^{(1)}\|_{\widetilde L^2(0,T; B^{\alpha + \sigma}_{2,\infty})} \le \sqrt T \|S_2 u_0\|_{B^{\alpha + \sigma}_{2,\infty}} \le \sqrt T\,C\, {\|u_0\|_{ B_{2,\infty}^{\sigma}}}
    \le M.
    \eeno
    Assuming that $(u^{(n)}, b^{(n)})$ obeys the bounds defined
    in $Y$, namely
    \beno
    \|u^{(n)}\|_{\widetilde L^\infty(0,T;  B^{\sigma}_{2,\infty})}\le M, \quad
    \|b^{(n)}\|_{\widetilde L^\infty(0,T;  B^{\sigma}_{2,\infty})}
    \le M,  \quad \|u^{(n)}\|_{\widetilde L^2(0,T; B^{\alpha + \sigma}_{2,\infty})} \le M,
    \eeno
    we prove that $(u^{(n+1)}, b^{(n+1)})$ obeys the same bound for sufficiently small $T>0$.

\vskip .1in
The proof involves inhomogeneous dyadic block operator $\Delta_j$ and  the inhomogeneous Besov spaces. Let $j\ge 0$ be an integer.
Applying $\Delta_j$ to the second and third equations in (\ref{sb1}) and
then dotting
by $(\Delta_j u^{(n+1)}, \Delta_j b^{(n+1)})$, we have
\beq\label{xj8}
 \frac{d}{dt} \left(\|\Delta_j u^{(n+1)}\|_{L^2}^2 + \|\Delta_j b^{(n+1)}\|_{L^2}^2\right)  + C_0\, 2^{2\alpha j}\, \|
\Delta_j u^{(n+1)}\|_{L^2}^2\le  E_1 + E_2 + E_3,
\eeq
where $C_0>0$ is constant and
\beno
&& E_1 = - 2 \int \Delta_j (u^{(n)}\cdot \na u^{(n+1)}) \cdot \Delta_j  u^{(n+1)}\, dx,\\
&& E_2 = - 2 \int \Delta_j (u^{(n)}\cdot \na b^{(n+1)}) \cdot \Delta_j  b^{(n+1)}\, dx,\\
&& E_3 = 2 \int \Delta_j (b^{(n)}\cdot \na b^{(n+1)})\cdot \Delta_j  u^{(n+1)}\,dx
+ \int \Delta_j (b^{(n)}\cdot \na u^{(n+1)})\cdot \Delta_j  b^{(n+1)}\,dx.
\eeno
According to Lemma \ref{para}, $E_1$ is bounded by
\beno
|E_1| &\le& C\, \|\Delta_j  u^{(n+1)}\|_{L^2}^2 \, \sum_{m\le j-1}
2^{(1+\frac{d}{2})m} \|\Delta_m u^{(n)}\|_{L^2} \notag\\
&& +\,C\, \|\Delta_j  u^{(n+1)}\|_{L^2}\,  \|\Delta_j  u^{(n)}\|_{L^2} \, \sum_{m\le j-1}
2^{(1+\frac{d}{2})m} \|\Delta_m u^{(n+1)}\|_{L^2} \notag\\
&& +\,C\, \|\Delta_j  u^{(n+1)}\|_{L^2}\,2^j\,  \sum_{k\ge j-1} 2^{\frac{d}{2} k} \,{\|\Delta_k u^{(n)}\|_{L^2} \,\|\widetilde{\Delta}_k u^{(n+1)}\|_{L^2}}\notag\\
&:=& L_1 + L_2 + L_3.
\eeno
$E_2$ is bounded  by
\beno
|E_2| &\le& C\, \|\Delta_j  b^{(n+1)}\|^2_{L^2}\,  \sum_{m\le j-1}
2^{(1+\frac{d}{2})m} \|\Delta_m u^{(n)}\|_{L^2} \\
&& +\,C\, \|\Delta_j  b^{(n+1)}\|_{L^2}\,  \|\Delta_j  u^{(n)}\|_{L^2} \, {\sum_{m\le j-1}}
2^{(1+\frac{d}{2})m} \|\Delta_m b^{(n+1)}\|_{L^2} \\
&& +\,C\, \|\Delta_j  b^{(n+1)}\|_{L^2}\,2^j\,  \sum_{k\ge j-1} 2^{\frac{d}{2} k} \,\|\widetilde{\Delta}_k b^{(n+1)}\|_{L^2}
\,\|\Delta_k u^{(n)}\|_{L^2}\\
&:=& L_4 +L_5 +L_6.
\eeno
$E_3$  is bounded by
\beno
|E_3| &\le& C\, \|\Delta_j  u^{(n+1)}\|_{L^2} \,\|\Delta_j  b^{(n+1)}\|_{L^2}\, \sum_{m\le j-1}
2^{(1+\frac{d}{2})m} \|\Delta_m b^{(n)}\|_{L^2} \\
&& +\,C\, \|\Delta_j  u^{(n+1)}\|_{L^2}\,  \|\Delta_j  b^{(n)}\|_{L^2} \, \sum_{m\le j-1}
2^{(1+\frac{d}{2})m} \|\Delta_m b^{(n+1)}\|_{L^2} \\
&& +\,C\, \|\Delta_j  u^{(n+1)}\|_{L^2}\,2^j\,  \sum_{k\ge j-1} 2^{\frac{d}{2} k} \,{\|\Delta_k b^{(n)}\|_{L^2} \,\|\widetilde{\Delta}_k b^{(n+1)}\|_{L^2}}\\
&& +\,C\, \|\Delta_j  b^{(n+1)}\|_{L^2}\,  \|\Delta_j  b^{(n)}\|_{L^2} \, {\sum_{m\le j-1}}
2^{(1+\frac{d}{2})m} \|\Delta_m u^{(n+1)}\|_{L^2} \\
&& +\,C\, \|\Delta_j  b^{(n+1)}\|_{L^2}\,2^j\,  \sum_{k\ge j-1} 2^{\frac{d}{2} k} \,\|\widetilde{\Delta}_k u^{(n+1)}\|_{L^2}
\,\|\Delta_k b^{(n)}\|_{L^2}\\
&:=&L_7 +L_8+L_9+L_{10}+L_{11}.
\eeno
Inserting these bounds in (\ref{xj8}) and then integrating in time yield
\ben
&& \|\Delta_j u^{(n+1)}\|_{L^2}^2 + \|\Delta_j b^{(n+1)}\|_{L^2}^2 +  C_0\, 2^{2\alpha j}\, \int_0^t \|
\Delta_j u^{(n+1)}\|_{L^2}^2\,d\tau  \notag\\
&\le&   \|\Delta_j u_0^{(n+1)}\|_{L^2}^2 + \|\Delta_j b_0^{(n+1)}\|_{L^2}^2 + \int_0^t (L_1 +\cdots + L_{11})\, d\tau. \label{xjp}
\een
Taking $L^\infty(0, T)$ of (\ref{xjp}), then multiplying by $2^{2 \sigma j}$ and
taking the sup in $j$ yield
\ben
&& \|u^{(n+1)}\|^2_{\widetilde L^\infty(0, T; B^{\sigma}_{2,\infty})} + \|b^{(n+1)}\|^2_{\widetilde L^\infty(0, T; B^{\sigma}_{2,\infty})} + C_0
\|u^{(n+1)}\|^2_{\widetilde L^2(0, T; B^{\sigma + \alpha}_{2,\infty})} \notag\\
&\le& \|u_0\|^2_{B^{\sigma}_{2,\infty}} + \|b_0\|^2_{B^{\sigma}_{2,\infty}}
+ \sup_{j} \,2^{2 \sigma j} \int_0^T (L_1 +\cdots +L_{11})\,d\tau.  \label{xjpp}
\een
We now estimate the eleven terms on the right. By H\"{o}lder's inequality,
\beno
&& \sup_{j} \,2^{2 \sigma j} \int_0^T L_1 \,d\tau \notag\\
&=& C\, \sup_{j} \,2^{2 \sigma j} \int_0^T \|\Delta_j  u^{(n+1)}\|_{L^2}^2 \, \sum_{m\le j-1}
2^{(1+\frac{d}{2})m} \|\Delta_m u^{(n)}\|_{L^2}\,d\tau \notag\\
&\le& C\, \|u^{(n+1)}\|^2_{\widetilde L^\infty(0, T; B^{\sigma}_{2,\infty})}
\sup_{j}  \sum_{m\le j-1}
2^{(1+\frac{d}{2}-(\alpha + \sigma))m} \, \int_0^T 2^{(\alpha + \sigma)m}\|\Delta_m u^{(n)}\|_{L^2}\,d\tau  \notag\\
&\le& C\, \|u^{(n+1)}\|^2_{\widetilde L^\infty(0, T; B^{\sigma}_{2,\infty})}
 \sum_{m\le j-1}
2^{(1+\frac{d}{2}-(\alpha + \sigma))m} \, \sqrt{T}\, \|2^{(\alpha + \sigma)m}\|\Delta_m u^{(n)}\|_{L^2}\|_{L^2(0,T)}  \notag\\
&\le& C\, \|u^{(n+1)}\|^2_{\widetilde L^\infty(0, T; B^{\sigma}_{2,\infty})}\,
\sqrt{T}\,\sup_{m} \|2^{(\alpha + \sigma)m}\|\Delta_m u^{(n)}\|_{L^2}\|_{L^2(0,T)}\notag\\
&=& C\,\sqrt{T}\, \|u^{(n)}\|_{\widetilde L^2(0, T; B^{\sigma + \alpha}_{2,\infty})}\, \|u^{(n+1)}\|^2_{\widetilde L^\infty(0, T; B^{\sigma}_{2,\infty})}\notag\\
& \le & C\,\sqrt{T}\, M \,\|u^{(n+1)}\|^2_{\widetilde L^\infty(0, T; B^{\sigma}_{2,\infty})},
\eeno
where we have used the fact that $\alpha + \sigma> 1+ \frac{d}{2}$ and we are working with inhomogeneous dyadic blocks.
The terms with $L_2$, $L_3$ and $L_4$ can be bounded very similarly and the bounds for them are
\beno
&& \sup_{j} \,2^{2 \sigma j} \int_0^T L_2 \,d\tau \le \, C\,\sqrt{T}\, M\, \,\|u^{(n+1)}\|_{\widetilde L^\infty(0, T; B^{\sigma}_{2,\infty})}\, \|u^{(n+1)}\|_{\widetilde L^2(0, T; B^{\sigma + \alpha}_{2,\infty})}, \\
&&  \sup_{j} \,2^{2 \sigma j} \int_0^T L_3 \,d\tau \le \, C\,\sqrt{T}\, M \,\|u^{(n+1)}\|^2_{\widetilde L^\infty(0, T; B^{\sigma}_{2,\infty})},\\
&& \sup_{j} \,2^{2 \sigma j} \int_0^T L_4 \,d\tau \le \, C\,\sqrt{T}\, M \,\|b^{(n+1)}\|^2_{\widetilde L^\infty(0, T; B^{\sigma}_{2,\infty})}.
\eeno
The term with $L_5$ is estimated slightly differently.
\beno
&& \sup_{j} \,2^{2 \sigma j} \int_0^T L_5 \,d\tau\\
&=& C\, \sup_{j} \,2^{2 \sigma j} \int_0^T \|\Delta_j  b^{(n+1)}\|_{L^2} \,\|\Delta_j u^{(n)}\|_{L^2}\,  \sum_{m\le j-1}
2^{(1+\frac{d}{2})m} \|\Delta_m b^{(n+1)}\|_{L^2}\,d\tau \\
&=& C\, \sup_{j} \int_0^T 2^{\sigma j} \|\Delta_j  b^{(n+1)}\|_{L^2}\,
2^{(\alpha + \sigma)j }\|\Delta_j u^{(n)}\|_{L^2}\,\\
&& \qquad\qquad\qquad \times \sum_{m\le j-1} 2^{\alpha(m-j)}\,
2^{(1+\frac{d}{2}- (\alpha+ \sigma))m} \,2^{\sigma m}\, \|\Delta_m b^{(n+1)}\|_{L^2}\,d\tau\\
&\le& C\, \|b^{(n+1)}\|^2_{\widetilde L^\infty(0, T; B^{\sigma}_{2,\infty})}
\sup_{j}  \,  \int_0^T 2^{(\alpha + \sigma)j}\|\Delta_j u^{(n)}\|_{L^2}\,d\tau\\
&\le& C\, \|b^{(n+1)}\|^2_{\widetilde L^\infty(0, T; B^{\sigma}_{2,\infty})}\,
\sqrt{T}\,\sup_{j} \|2^{(\alpha + \sigma)j}\|\Delta_j u^{(n)}\|_{L^2}\|_{L^2(0,T)}\\
&=& C\,\sqrt{T}\, \|u^{(n)}\|_{\widetilde L^2(0, T; B^{\sigma + \alpha}_{2,\infty})}\, \|b^{(n+1)}\|^2_{\widetilde L^\infty(0, T; B^{\sigma}_{2,\infty})}\\
& \le & C\,\sqrt{T}\, M \,\|b^{(n+1)}\|^2_{\widetilde L^\infty(0, T; B^{\sigma}_{2,\infty})},
\eeno
where we have used the fact that $m-j <0$ and $1+\frac{d}{2} -\alpha-\sigma<0$.
The estimates of the other terms are similar,
\beno
&& \sup_{j} \,2^{2 \sigma j} \int_0^T L_6 \,d\tau \le \,C\,\sqrt{T}\, M \,\|b^{(n+1)}\|^2_{\widetilde L^\infty(0, T; B^{\sigma}_{2,\infty})},\\
&& \sup_{j} \,2^{2 \sigma j} \int_0^T L_7 \,d\tau \le \,C\,\sqrt{T}\, M\,
\|b^{(n+1)}\|_{\widetilde L^\infty(0, T; B^{\sigma}_{2,\infty})}\,
\|u^{(n+1)}\|_{\widetilde L^2(0, T; B^{\sigma + \alpha}_{2,\infty})}, \\
&& \sup_{j} \,2^{2 \sigma j} \int_0^T L_8 \,d\tau \le \,C\,\sqrt{T}\, M\,
\|b^{(n+1)}\|_{\widetilde L^\infty(0, T; B^{\sigma}_{2,\infty})}\,
\|u^{(n+1)}\|_{\widetilde L^2(0, T; B^{\sigma + \alpha}_{2,\infty})}, \\
&& \sup_{j} \,2^{2 \sigma j} \int_0^T L_9 \,d\tau \le \,C\,\sqrt{T}\, M\,
\|b^{(n+1)}\|_{\widetilde L^\infty(0, T; B^{\sigma}_{2,\infty})}\,
\|u^{(n+1)}\|_{\widetilde L^2(0, T; B^{\sigma + \alpha}_{2,\infty})}, \\
&& \sup_{j} \,2^{2 \sigma j} \int_0^T L_{10} \,d\tau \le \,C\,\sqrt{T}\, M\,
\|b^{(n+1)}\|_{\widetilde L^\infty(0, T; B^{\sigma}_{2,\infty})}\,
\|u^{(n+1)}\|_{\widetilde L^2(0, T; B^{\sigma + \alpha}_{2,\infty})}, \\
&& \sup_{j} \,2^{2 \sigma j} \int_0^T L_{11} \,d\tau \le \,C\,\sqrt{T}\, M\,
\|b^{(n+1)}\|_{\widetilde L^\infty(0, T; B^{\sigma}_{2,\infty})}\,
\|u^{(n+1)}\|_{\widetilde L^2(0, T; B^{\sigma + \alpha}_{2,\infty})}.
\eeno
Inserting the bounds above in (\ref{xjpp}) yields
\ben
&& \|(u^{(n+1)}, b^{(n+1)})\|^2_{\widetilde L^\infty(0, T; B^{\sigma}_{2,\infty})}
+ C_0
\|u^{(n+1)}\|^2_{\widetilde L^2(0, T; B^{\sigma + \alpha}_{2,\infty})} \notag\\
&\le& \|(u_0, b_0)\|^2_{B^{\sigma}_{2,\infty}}
+ C\,\sqrt{T}\,M \,\|(u^{(n+1)}, b^{(n+1)})\|^2_{\widetilde L^\infty(0, T; B^{\sigma}_{2,\infty})} \notag \\
&& +\, C\,\sqrt{T}\, M\, \|(u^{(n+1)}, b^{(n+1)})\|_{\widetilde L^\infty(0, T; B^{\sigma}_{2,\infty})}\, \|u^{(n+1)}\|_{\widetilde L^2(0, T; B^{\sigma + \alpha}_{2,\infty})}. \label{keyes}
\een
We choose $T>0$ to be sufficiently small such that
\beno
&& C\,\sqrt{T}\,M \,\|(u^{(n+1)}, b^{(n+1)})\|^2_{\widetilde L^\infty(0, T; B^{\sigma}_{2,\infty})} \\
&& +\, C\,\sqrt{T}\, M\, \|(u^{(n+1)}, b^{(n+1)})\|_{\widetilde L^\infty(0, T; B^{\sigma}_{2,\infty})}\, \|u^{(n+1)}\|_{\widetilde L^2(0, T; B^{\sigma + \alpha}_{2,\infty})}\\
&\le& \frac34\|(u^{(n+1)}, b^{(n+1)})\|^2_{\widetilde L^\infty(0, T; B^{\sigma}_{2,\infty})} + \frac34 C_0\,\|u^{(n+1)}\|^2_{\widetilde L^2(0, T; B^{\sigma + \alpha}_{2,\infty})}.
\eeno
Then (\ref{keyes}) implies
\beno
&& \|(u^{(n+1)}, b^{(n+1)})\|_{\widetilde L^\infty(0, T; B^{\sigma}_{2,\infty})} \le 2 \|(u_0, b_0)\|_{B^{\sigma}_{2,\infty}} = M, \\
&& \|u^{(n+1)}\|_{\widetilde L^2(0, T; B^{\sigma + \alpha}_{2,\infty})}  \le \frac{2}{\sqrt{C_0}} \|(u_0, b_0)\|_{B^{\sigma}_{2,\infty}} \le M.
\eeno
These uniform bounds allow us to extract a weakly convergent subsequence. That is,
there is $(u, b)\in Y$ such that a subsequence of $(u^{(n)}, b^{(n)})$ (still denoted by $(u^{(n)}, b^{(n)})$) satisfies
\beno
&& u^{(n)}  \overset{\ast}{\rightharpoonup} u \quad \mbox{in}\quad \widetilde L^\infty(0,T; B^{\sigma}_{2,\infty}) \cap
\widetilde L^2(0, T; B^{\sigma + \alpha}_{2,\infty}), \\
&& b^{(n)}  \overset{\ast}{\rightharpoonup} b \quad \mbox{in}\quad \widetilde L^\infty(0,T; B_{2,\infty}^{\sigma}).
\eeno
In order to show that $(u, b)$ is a weak solution of (\ref{mhd}), we need to further
extract a subsequence which converges strongly to $(u, b)$. This is done via the Aubin-Lions Lemma. We can show by making use of the equations in (\ref{sb}) that $(\p_t u^{(n)},\p_t b^{(n)})$ is {uniformly} bounded in
\beno
 \p_t u^{(n)} \in
\widetilde L^2(0, T; B^{\sigma -\alpha }_{2,\infty}), \qquad
\p_t b^{(n)}\in\widetilde L^2(0, T;  B^{\frac{d}{2}}_{2,\infty}).
\eeno
Since the domain here is the whole space $\mathbb R^d$, we need to combine Cantor's diagonal process with the Aubin-Lions Lemma to show that a subsequence of the weakly convergent subsequence, still denoted by $(u^{(n)}, b^{(n)})$, has the following strongly convergent property,
\beno
&& u^{(n)} \to u \qquad \mbox{in} \quad L^2(0, T; B^{\gamma_1}_{2,\infty}(Q))\quad \mbox{for}\quad \gamma_1 \in (\sigma -\alpha, \sigma + \alpha)\\
&& b^{(n)} \to b \qquad \mbox{in} \quad L^2(0, T; B^{\gamma_2}_{2,\infty}(Q))\quad \mbox{for}\quad \gamma_2 \in (d/2, \sigma),
\eeno
where $Q\subset \mathbb R^d$ is any compact subset. This strong convergence property would allow us to show that $(u, b)$ is indeed a weak solution of (\ref{mhd}). This process is routine and we omit the details. This completes the proof for the existence part of Theorem \ref{main}
in the case when $\alpha < 1$.
\end{proof}

\vskip .3in
\section{Proof for the uniqueness part of Theorem \ref{main}}
\label{uni}

This section proves the uniqueness part of Theorem \ref{main}.

\begin{proof}
	Assume that $(u^{(1)}, b^{(1)})$ and $(u^{(2)}, b^{(2)})$ are two solutions. Their difference $(\widetilde u, \widetilde b)$ with
	$$
	\widetilde u = u^{(2)} - u^{(1)}, \qquad \widetilde b = b^{(2)} - b^{(1)}
	$$
	satisfies
	\beq \label{mhd_diff}
	\begin{cases}
		\p_t \widetilde u + \nu (-\Delta)^\alpha \widetilde u = -\mathbb P (u^{(2)}\cdot\na \widetilde u +\widetilde u\cdot \na u^{(1)} ) + \mathbb P (b^{(2)} \cdot\na \widetilde b + \widetilde b\cdot\na b^{(1)}), \\
		\p_t\widetilde b  = -u^{(2)}\cdot\na \widetilde b -\widetilde u\cdot\na b^{(1)}  + b^{(2)} \cdot\na \widetilde u + \widetilde b\cdot\na u^{(1)},  \\
		\na \cdot \widetilde u =\na \cdot\widetilde b =0, \\
		\widetilde u(x,0) =0, \quad\widetilde b(x,0) =0.
	\end{cases}
	\eeq
We estimate the difference
$(\widetilde u, \widetilde b)$ in $L^2(\mathbb R^d)$. Dotting (\ref{mhd_diff}) by ${(\widetilde u, \widetilde b)}$ and applying the divergence-free condition, we find
	\ben\label{xdd}
	\frac12 \frac{d}{dt} (\|\widetilde{u}\|_{L^2}^2 + \|\widetilde{b}||_{L^2}^2) + \nu \|\Lambda^\alpha \widetilde{u}\|_{L^2}^2 = Q_1 + Q_2 + Q_3 + Q_4 + Q_5,
	\een
	where
	\beno
	&& Q_1 = - \int \widetilde u\cdot \na u^{(1)} \cdot \widetilde u\,dx,\\
	&& Q_2 = \int b^{(2)} \cdot\na \widetilde b \cdot \widetilde u\,dx + \int b^{(2)} \cdot\na \widetilde u \cdot \widetilde{b}\, dx,\\
	&&  Q_3 = \int \widetilde b\cdot\na b^{(1)}\cdot \widetilde{u}\, dx, \\
	&& Q_4 = - \int \widetilde u\cdot\na b^{(1)}\cdot \widetilde{b}\, dx,\\
	&& Q_5 = \int \widetilde b\cdot\na u^{(1)} \cdot \widetilde{b}\, dx.
	\eeno
	Due to $\na \cdot b^{(2)} =0$, we find $Q_2 =0$ after integration by parts. We remark that $Q_3 + Q_4$ is not necessarily zero. The rest of the proof distinguish between the two cases: $\alpha\ge 1$ and $\alpha <1$. For the sake of clarity, we divide the rest of this section into two subsections.

	\subsection{The case $\alpha\ge 1$} The uniqueness for the case when $\alpha=1$  has been obtained in \cite{CMRR,LTY,Wan}.  Our attention will be focused on $\alpha>1$. In this subsection $\Delta_j$ denotes the homogeneous dyadic block operators for the simplicity of notation.
	By H\"{o}lder's inequality,
	\beno
	|Q_1| \le \|\na  u^{(1)}\|_{L^\infty} \, \|\widetilde{u}\|_{L^2}^2, \qquad |Q_5| \le  \,C\,  \|\na  u^{(1)}\|_{L^\infty} \, \|\widetilde{b}\|_{L^2}^2.
	\eeno
	By integration by parts,
	$$
	Q_3 = -\int \widetilde b\cdot\na \widetilde{u}\cdot  b^{(1)}\,dx.
	$$
	For $p$ and $q$ defined by
	$$
	\frac1p = \frac12+ \frac{1-\alpha}{d}, \qquad \frac1p + \frac1q = \frac12,
	$$
	we have, by H\"{o}lder's inequality,
	\beno
	|Q_3| &\le& \|\widetilde b\|_{L^2} \, \|\na \widetilde{u}\|_{L^p}\, \|b^{(1)}\|_{L^q} \\
	&\le& \,C\, \|\widetilde b\|_{L^2} \,\|\Lambda^{\alpha-1} \na \widetilde{u}\|_{L^2}\, \|b^{(1)}\|_{L^q}\\
	&\le& \frac{\nu}{4} \|\Lambda^\alpha \widetilde{u}\|_{L^2}^2 + \,C\, \|b^{(1)}\|^2_{L^q}\,  \|\widetilde b\|^2_{L^2}.
	\eeno
	$Q_4$ obeys the same bound. Inserting these bounds in (\ref{xdd}), we find
	\beno
	&&\frac{d}{dt} (\|\widetilde{u}\|_{L^2}^2 + \|\widetilde{b}||_{L^2}^2) + \nu \|\Lambda^\alpha \widetilde{u}\|_{L^2}^2 \\
	&\le&\, C\, \|\na  u^{(1)}\|_{L^\infty}\,
	(\|\widetilde{u}\|_{L^2}^2 + \|\widetilde{b}||_{L^2}^2) + \,C\, \|b^{(1)}\|^2_{L^q}\,  \|\widetilde b\|_{L^2}.
	\eeno
	By Bernstein's inequality,
	$$
	\|\na u^{(1)}\|_{L^\infty}  \le \sum_j \|\Delta_j \na u^{(1)}\|_{L^\infty} \le \sum_j 2^{(1+\frac{d}{2})j } \|\Delta_j u^{(1)}\|_{L^2} = \|u^{(1)}\|_{\mathring B^{1+\frac{d}{2}}_{2,1}}
	$$
	and
	\beno
	\|b^{(1)}\|_{L^q} &\le& \sum_j \|\Delta_j  b^{(1)}\|_{L^q} \le C\,\sum_{j}
	2^{j d (\frac12 -\frac1q)} \|\Delta_j b^{(1)}\|_{L^2} \\
	&=& C\,\sum_{j} 2^{j d(\frac12+\frac{1-\alpha}{d})} \|\Delta_j b^{(1)}\|_{L^2} = C\, \|b^{(1)}\|_{\mathring B^{1+\frac{d}{2}-\alpha}_{2,1}}.
	\eeno
	Therefore,
	\beno
	\frac{d}{dt} (\|\widetilde{u}\|_{L^2}^2 + \|\widetilde{b}\|_{L^2}^2)  \le \,C\, (\|u^{(1)}\|_{\mathring B^{1+\frac{d}{2}}_{2,1}} + \|b^{(1)}\|^2_{\mathring B^{1+\frac{d}{2}-\alpha}_{2,1}})\,  (\|\widetilde{u}\|_{L^2}^2 + \|\widetilde{b}||_{L^2}^2).
	\eeno
	Due to the time integrability
	$$
	\int_0^T (\|u^{(1)}\|_{\mathring B^{1+\frac{d}{2}}_{2,1}} + \|b^{(1)}\|^2_{\mathring B^{1+\frac{d}{2}-\alpha}_{2,1}}) \,dt <\infty,
	$$
	Gronwall's inequality then implies that
	$$
	\|\widetilde{u}\|_{L^2} =  \|\widetilde{b}\|_{L^2} =0.
	$$

\vskip .1in
\subsection{The case $\alpha <1$} The operator $\Delta_j$ in this subsection denotes
the inhomogeneous dyadic block operators and the Besov spaces are inhomogeneous. By H\"{o}lder's inequality,
	$$
	|Q_1| \le \|\na  u^{(1)}\|_{L^\infty} \, \|\widetilde{u}\|_{L^2}^2, \qquad
|Q_5| \le 	 \|\na  u^{(1)}\|_{L^\infty} \, \|\widetilde{b}\|_{L^2}^2.
$$
To bound $Q_3$, we set
	$$
	\frac1p = \frac12 -\frac{\alpha}{d}, \qquad \frac1q +\frac1p =\frac12
	$$
	and apply H\"{o}lder's inequality and the Hardy-Littlewood-Sobolev inequality to obtain
    \beno
	|Q_3| &\le& \|\widetilde b\|_{L^2}\, \|\na  b^{(1)}\|_{L^q}\, \|\widetilde{u}\|_{L^p}\\
	&\le& \,C\, \|\widetilde b\|_{L^2}\,\|\na b^{(1)}\|_{L^q}\, \|\Lambda^\alpha \widetilde u\|_{L^2}\\
	&\le& \, \frac{\nu}{4} \|\Lambda^\alpha \widetilde u\|^2_{L^2} + C\, \|\na b^{(1)}\|_{L^q}^2 \, \|\widetilde b\|^2_{L^2}.
	\eeno
	$Q_4$ obeys exactly the same bound. Inserting these bounds in (\ref{xdd}), we find
	\ben
	&& \frac{d}{dt} (\|\widetilde{u}\|_{L^2}^2 + \|\widetilde{b}||_{L^2}^2) + \nu \|\Lambda^\alpha \widetilde{u}\|_{L^2}^2 \notag\\
	&\le& \, C\, \|\na  u^{(1)}\|_{L^\infty}\,
	(\|\widetilde{u}\|_{L^2}^2 + \|\widetilde{b}||_{L^2}^2) + \,C\, \|\na b^{(1)}\|^2_{L^q}\,  \|\widetilde b\|^2_{L^2}. \label{ggp}
	\een
	By Bernstein's inequality,
	\ben
	\int_0^T \|\na  u^{(1)}\|_{L^\infty}\, dt &\le& \sum_{j \ge -1} \int_0^T \|\Delta_j \na u^{(1)}\|_{L^\infty}\,dt \notag\\
	&\le& \sum_{j \ge -1} \int_0^T 2^{(1+\frac{d}{2})j } \|\Delta_j u^{(1)}\|_{L^2} dt \notag\\
	&\le& \sum_{j \ge -1} 2^{(1+\frac{d}{2}-\alpha -\sigma)j } \int_0^T   2^{(\alpha + \sigma)j}\, \|\Delta_j u^{(1)}\|_{L^2} dt \notag\\
		&\le& \sum_{j \ge -1} 2^{(1+\frac{d}{2}-\alpha -\sigma)j } \sqrt{T}\,
		\|2^{(\alpha + \sigma)j}\, \|\Delta_j u^{(1)}\|_{L^2}\|_{L^2(0, T)} \notag\\
		&\le& \le\,C\, \sqrt{T}\, \|u^{(1)}\|_{\widetilde L^2(0, T; B^{\sigma + \alpha}_{2,\infty})}, \label{jpp0}
	\een
	where we have used the fact that $\sigma> 1+ \frac{d}{2} -\alpha$.
	In addition,
	\beno
	\|\na  b^{(1)}\|_{L^q} &\le& \sum_{j\ge -1} \|\Delta_j \na b^{(1)}\|_{L^q} \\
	&\le& \, C\, \sum_{j\ge -1} 2^{j + dj (\frac12-\frac1q)}\,\|\Delta_j b^{(1)}\|_{L^2}\\
	&\le& \, C\,\sum_{j\ge -1} 2^{j + dj (\frac12-\frac\alpha d)}\, \|\Delta_j b^{(1)}\|_{L^2}\\
	&=& \, C\,\sum_{j\ge -1} 2^{(1+ \frac{d}{2} -\alpha -\sigma)j }\,
	2^{\sigma j} \,\|\Delta_j b^{(1)}\|_{L^2}\\
	&\le& \, C\,\|b^{(1)}\|_{B^{\sigma}_{2,\infty}},
	\eeno
  	where again we have used the fact that $\sigma> 1+ \frac{d}{2} -\alpha$. Therefore,
  	\ben
  		\int_0^T  \|\na  b^{(1)}\|_{L^q}^2\,dt \le \,C\, T \, \|b^{(1)}\|^2_{\widetilde L^\infty(0, T; B^{\sigma}_{2,\infty})}. \label{jpp}
  	\een
	Applying Gronwall's inequality to (\ref{ggp}) and invoking (\ref{jpp0}) and (\ref{jpp}), we obtain
	$$
	\|\widetilde{u}\|_{L^2} =  \|\widetilde{b}\|_{L^2} =0,
	$$
	which leads to the desired uniqueness. This completes the proof of the uniqueness part  of Theorem \ref{main}.
\end{proof}

\vskip .3in
\section{Conclusion and discussions}
\label{diss}

We have established that, for $\alpha\ge 1$,
any initial data $(u_0, b_0)$ with
\beno
u_0 \in \mathring B_{2,1}^{\frac{d}{2}+ 1 -2\alpha} (\mathbb R^d), \quad  b_0\in \mathring B_{2,1}^{\frac{d}{2}+1 -\alpha} (\mathbb R^d),
\eeno
and, for $\alpha<1$, any initial data $(u_0, b_0)$ with
\ben
&&  u_0 \in B_{2,\infty}^{\sigma} (\mathbb R^d), \quad b_0\in B_{2,\infty}^{\sigma}, (\mathbb R^d), \qquad \sigma > \frac{d}{2}+1 -\alpha \label{ddyou}
\een
leads to a unique local weak solution of (\ref{mhd}). This purpose of this section is to explain in some detail on why these regularity assumptions may be optimal. The optimality for the case $\alpha\ge 1$ can be easily explained. The index  $\frac{d}{2}+ 1 -2\alpha$ is minimal for the velocity in order to achieve
the uniqueness. As we know, the velocity $u$ should obey $\int_0^T \|\na u\|_{L^\infty}\,dt<\infty$ or a slightly weaker version in order to guarantee the uniqueness. In the Besov setting here, we need
$$
\|u\|_{L^1(0, T; \mathring B^{1+\frac{d}{2}}_{2,1}(\mathbb R^d))} <\infty,
$$
which, in turn, requires that
$$
u \in \widetilde L^\infty(0, T; \mathring B^{1+\frac{d}{2}-2\alpha}_{2,1}(\mathbb R^d)).
$$
This is how the index $\frac{d}{2}+ 1 -2\alpha$ arises.  Once the Besov space for $u_0$ is set, the functional setting for $b_0$ is determined correspondingly.

\vskip .1in
We now  explain why the initial setup for the case $\alpha<1$ may be optimal.
We have attempted to replace (\ref{ddyou}) by several weaker assumptions, but we failed to establish the desired existence and uniqueness. We now describe the difficulties associated with those weaker initial data.

\subsection{Can we replace (\ref{ddyou}) by $u_0 \in B_{2,1}^{\frac{d}{2}+ 1 -2\alpha} (\mathbb R^d)$ and $b_0\in  B_{2,1}^{\frac{d}{2}} (\mathbb R^d)$?} We would have difficulty proving the uniform boundedness of the successive approximation sequence in the existence proof part. If we assume that
\beno
u_0 \in B_{2,1}^{\frac{d}{2}+ 1 -2\alpha} (\mathbb R^d) \quad\mbox{and}\quad  b_0\in  B_{2,1}^{\frac{d}{2}} (\mathbb R^d),
\eeno
then the corresponding functional space for $(u, b)$ would be
\beno
Y \equiv  \Big\{(u, b) \big|&& \|u\|_{\widetilde L^\infty(0,T;  B^{\frac{d}{2}+ 1 -2\alpha}_{2,1})} \le M, \quad  \|b\|_{\widetilde L^\infty(0,T;  B_{2,1}^{\frac{d}{2}})}
\le M, \,\,\, \notag\\
&& \|u\|_{L^1(0, T; \,  B^{\frac{d}{2}+ 1}_{2,1})} \le \delta, \quad  \|u\|_{\widetilde L^2(0,T;  B^{\frac{d}{2}+ 1-\alpha}_{2,1})} \le \delta\Big\}
\eeno
Suppose we construct the successive approximation sequence by (\ref{sb}). We can  obtain suitable bounds for
$$
\|u^{(n+1)}\|_{\widetilde L^\infty(0,T;  B^{\frac{d}{2}+ 1 -2\alpha}_{2,1})}, \quad
\|b^{(n+1)}\|_{\widetilde L^\infty(0,T;  B_{2,1}^{\frac{d}{2}})}.
$$
We would have difficulty controlling  $\|u^{(n+1)}\|_{L^1(0, T; \,  B^{\frac{d}{2}+ 1}_{2,1})}$ due to the term $b^{(n)} \cdot\na b^{(n)}$ in (\ref{sb}). A quick way to see the difficultly is to count the derivatives needed and the derivatives allowed,
$$
\left(\frac{d}{2}+ 1\right) + \left(\frac{d}{2}+ 1\right) -2\alpha =2 \left(\frac{d}{2}+ 1 -\alpha\right) > 2 \cdot \frac{d}{2}.
$$
We explain the meaning of this inequality. The left-hand side $2 \left(\frac{d}{2}+ 1 -\alpha\right)$ represents the derivative imposed and the right-hand side $2 \cdot \frac{d}{2}$ denotes the derivatives allowed on the two $b^{(n)}$'s.  The first $\frac{d}{2}+ 1$ comes from $B^{\frac{d}{2}+ 1}_{2,1}$, the second $\frac{d}{2}+ 1$ represents the derivative when we estimates $\|b^{(n)} \cdot\na b^{(n)}\|_{L^2}$ and $-2\alpha$ is due to the dissipation. When $\alpha<1$, the derivatives imposed are more than the derivatives allowed and we can not close the estimates in $Y$.

\vskip .1in
\subsection{Can we replace (\ref{ddyou}) by $u_0 \in B_{2,1}^{\frac{d}{2}+ 1 -2\alpha} (\mathbb R^d)$ and $b_0\in  B_{2,1}^{\frac{d}{2} + 1-\alpha} (\mathbb R^d)$?} Even though we increased the regularity of $b_0$ to the level that would allow us to overcome one difficulty mentioned in the previous subsection, we would still have trouble proving the uniform boundedness of the successive approximation sequence in the existence proof part. If we assume that
\beno
u_0 \in B_{2,1}^{\frac{d}{2}+ 1 -2\alpha} (\mathbb R^d) \quad\mbox{and}\quad  b_0\in  B_{2,1}^{\frac{d}{2} + 1-\alpha} (\mathbb R^d),
\eeno
then the corresponding functional space for $(u, b)$ would be
\beno
Y \equiv  \Big\{(u, b) \big|&& \|u\|_{\widetilde L^\infty(0,T;  B^{\frac{d}{2}+ 1 -2\alpha}_{2,1})} \le M, \quad  \|b\|_{\widetilde L^\infty(0,T;  B_{2,1}^{\frac{d}{2} + 1-\alpha})}
\le M, \,\,\, \notag\\
&& \|u\|_{L^1(0, T; \,  B^{\frac{d}{2}+ 1}_{2,1})} \le \delta, \quad  \|u\|_{\widetilde L^2(0,T;  B^{\frac{d}{2}+ 1-\alpha}_{2,1})} \le \delta\Big\}
\eeno
Suppose we construct the successive approximation sequence by (\ref{sb}). We can  obtain suitable bounds for
$$
\|u^{(n+1)}\|_{\widetilde L^\infty(0,T;  B^{\frac{d}{2}+ 1 -2\alpha}_{2,1})}, \quad
\|u^{(n+1)}\|_{L^1(0, T; \,  B^{\frac{d}{2}+ 1}_{2,1})}.
$$
But this new setup would make it impossible to control
$$
\|b^{(n+1)}\|_{\widetilde L^\infty(0,T;  B_{2,1}^{\frac{d}{2}+1-\alpha})}.
$$
The difficulty comes from bounding the term $b^{(n)} \cdot\na u^{(n)}$ in the equation of $b^{(n+1)}$ in (\ref{sb}). In order to bound $\|\Delta_j (b^{(n)} \cdot\na u^{(n)})\|_{L^2}$, one naturally decomposes it by paraproducts as in
(\ref{xx}),
\beno
\|\Delta_j (b^{(n)} \cdot\na u^{(n)})\|_{L^2} &\le &
\,C\, 2^j\, \|\Delta_j  u^{(n)}\|_{L^2} \, \sum_{m\le j-1}
2^{\frac{d}{2}\,m} \|\Delta_m b^{(n)}\|_{L^2} \\
&& + \,C\,  \|\Delta_j  b^{(n)}\|_{L^2} \, \sum_{m\le j-1}
2^{(1+\frac{d}{2})m} \|\Delta_m u^{(n)}\|_{L^2}\\
&& + \,C\, 2^j\,  \sum_{k\ge j-1} 2^{\frac{d}{2} k} \,\|\Delta_k b^{(n)}\|_{L^2} \,\|\widetilde{\Delta}_k u^{(n)}\|_{L^2}.
\eeno
The trouble arises in the first  term on the right-hand side. For $\alpha<1$, we can no longer bound
$$
2^{(1-\alpha) j} \,\sum_{m\le j-1}
2^{\frac{d}{2}\,m} \|\Delta_m b^{(n)}\|_{L^2}
$$
by
$$
\sum_{m\le j-1}
2^{(\frac{d}{2}+1 -\alpha)\,m} \|\Delta_m b^{(n)}\|_{L^2}
$$
and, as a consequence, we are not able to control $b^{(n)} \cdot\na u^{(n)}$ by the desired bound $\|u^{(n)}\|_{L^1(0, T; \,  B^{\frac{d}{2}+ 1}_{2,1})}\,
\|b^{(n)}\|_{\widetilde L^\infty(0,T;  B_{2,1}^{\frac{d}{2}+1-\alpha})}$. This problem arises when $u$ and $b$ are in different functional settings. We can no longer estimate $u$ and $b$ simultaneously and the good structure of combining the terms $b\cdot\na b$ and $b\cdot\na u$ can no longer be taken advantage of.

\vskip .1in
\subsection{Can we replace (\ref{ddyou}) by $u_0 \in B_{2,1}^{\frac{d}{2}+ 1 -\alpha} (\mathbb R^d)$ and $b_0\in  B_{2,1}^{\frac{d}{2} + 1-\alpha} (\mathbb R^d)$?}  Even $u_0$ and $b_0$ are now in the same functional setting, but we are still not able to prove the uniform boundedness of the successive approximation sequence in the existence proof part. We now explain the difficulty. Naturally the
corresponding functional setting for $(u, b)$ is
 \beno
 Y \equiv  \Big\{(u, b) \big|&& \|(u, b)\|_{\widetilde L^\infty(0,T;  B^{\frac{d}{2}+ 1 -\alpha}_{2,1})} \le M,  \notag\\
 && \|u\|_{L^1(0, T; \,  B^{\frac{d}{2}+ 1}_{2,1})} \le \delta, \quad  \|u\|_{\widetilde L^2(0,T;  B^{\frac{d}{2}+ 1-\alpha}_{2,1})} \le \delta\Big\}
 \eeno
Suppose we construct the successive approximation sequence by (\ref{sb1}). In order
to make use of the cancellation in the combination of $b\cdot\na b$
and $b\cdot\na u$, we have to add the estimates at the $L^2$-level as in (\ref{xj8}). However, if we add them at the $L^2$-level, it is then impossible to
control the norm of $(u, b)$ in $B^{\frac{d}{2}+ 1 -\alpha}_{2,1}$. This is exactly
why we have selected the functional setting $B^{\sigma}_{2,\infty}$ with
$\sigma> \frac{d}{2}+ 1 -\alpha$ when $\alpha<1$, as in the proof in Section \ref{exi2}.

\vskip .1in
In conclusion, the regularity assumptions on the initial data in Theorem \ref{main} may be optimal.

\vskip .3in
\appendix

\section{Besov spaces and related tools}


This appendix provides the definition of the Besov spaces and related facts that have been used in the previous sections. Some of the materials are taken from \cite{BCD}. More details can be found in several books and many papers (see, e.g., \cite{BCD,BL,MWZ,RS,Tri}). In addition, we also prove several bounds on triple products involving Fourier localized functions. These bounds have been used in the previous sections.

\vskip .1in
We start with the partition of unit. Let $B(0, r)$ and ${\mathcal C}(0, r_1, r_2)$ denote the standard ball and the annulus, respectively,
$$
B(0, r) = \left\{\xi \in \mathbb R^d: \, |\xi| \le r \right\}, \qquad
\mathcal C (0, r_1, r_2) = \left\{\xi \in \mathbb R^d:\, r_1 \le|\xi|\le r_2 \right\}.
$$
There are two compactly supported smooth radial functions $\phi$ and $\psi$ satisfying
\ben
&& \mbox{supp} \,\phi \subset B(0, 4/3), \quad \mbox{supp} \,\psi \subset \mathcal C(0, 3/4, 8/3),  \notag\\
&& \phi(\xi) + \sum_{j\ge 0}  \psi(2^{-j} \xi) = 1 \qquad \mbox{for all}\,\, \xi \in \mathbb R^d,\label{aa}\\
&& \sum_{j\in \mathbb Z} \psi(2^{-j} \xi) = 1 \qquad \mbox{for $\xi \in \mathbb R^d\setminus \{0\}$ }.\notag
\een
We use $\widetilde h$ and $h$ to denote the inverse Fourier transforms of $\phi$ and
$\psi$ respectively,
$$
\widetilde h = \mathcal F^{-1} \phi, \quad  h = \mathcal F^{-1} \psi.
$$
In addition, for notational convenience, we write $\psi_j(\xi) = \psi(2^{-j} \xi)$. By a simple property of the Fourier transform,
$$
h_j(x) :=\mathcal F^{-1} (\psi_j)(x) = 2^{d j} \, h(2^j x).
$$
The inhomogeneous dyadic block operator $\Delta_j$ are defined as follows
\beno
&& \Delta_j f=0 \qquad \mbox{for $j\le -2$},\\
&& \Delta_{-1} f = \widetilde h \ast f = \int_{\mathbb R^d} f(x-y) \, \widetilde h(y)\,dy,\\
&& \Delta_j f = h_j \ast f = 2^{d j} \int_{\mathbb R^d} f(x-y) \,  h(2^j y)\,dy \qquad \mbox{for $j\ge 0$}.
\eeno
The corresponding inhomogeneous low frequency cut-off operator $S_j$ is defined by
$$
S_j f = \sum_{k\le j-1} \Delta_k f.
$$
For any function $f$ in the usual Schwarz class $\mathcal{S}$, (\ref{aa}) implies
\ben
\widehat f (\xi) = \phi(\xi)\, \widehat f (\xi) + \sum_{j\ge 0}  \psi(2^{-j} \xi)\,\widehat f (\xi)  \label{bb}
\een
or, in terms of the inhomogeneous dyadic block operators,
$$
f= \sum_{j\ge -1} \Delta_j f \quad \mbox{or}\quad \mbox{Id} = \sum_{j\ge -1} \Delta_j,
$$
where Id denotes the identity operator. More generally, for any $F$ in the space of tempered distributions, denoted  ${\mathcal S}'$, (\ref{bb}) still holds but in the distributional sense. That is, for $F \in {\mathcal S}'$,
\ben
F = \sum_{j\ge -1} \Delta_j F \quad \mbox{or}\quad \mbox{Id} = \sum_{j\ge -1} \Delta_j \qquad \mbox{in} \quad  {\mathcal S}'.  \label{bb1}
\een
In fact, one can verify that
$$
S_j F := \sum_{k\le j-1} \Delta_k F  \quad \to \quad F \qquad \mbox{in} \quad  {\mathcal S}'.
$$
(\ref{bb1}) is referred to as the Littlewood-Paley decomposition for tempered distributions.

\vskip .1in
In terms of the inhomogeneous dyadic block operators, we can write the standard
product in terms of the paraproducts, namely the Bony decomposition,
$$
F\, G = {\sum_{|j-k|\le 2} S_{k-1} F\, \Delta_k G + \sum_{|j-k|\le 2} \Delta_k F\, S_{k-1} G}
+ \sum_{k\ge j-1} \Delta_k F\, \widetilde \Delta_k G,
$$
where  $\widetilde\Delta_k = \Delta_{k-1} + \Delta_k + \Delta_{k+1}$.

\vskip .1in
The inhomogeneous Besov space can be defined in terms of  $\Delta_j$ specified above.
\begin{define}
	The inhomogeneous Besov space $B^s_{p,q}$ with $1\le p,q \le \infty$
	and $s\in {\mathbb R}$ consists of  $f\in {\mathcal S}'$
satisfying
$$
\|f\|_{B^s_{p,q}} \equiv \|2^{js} \|\Delta_j f\|_{L^p} \|_{l^q}
<\infty.
$$
\end{define}

\vskip .1in
The concepts defined above have their homogeneous version. The homogeneous dyadic block  and the homogeneous low frequency cutoff operators are defined by, for any $j\in \mathbb Z$,
\beno
&& \mathring \Delta_j f  = h_j \ast f = 2^{d j} \int_{\mathbb R^d} f(x-y) \,  h(2^j y)\,dy,\\
&& \mathring S_j f = \sum_{k\le j-1} \mathring \Delta_k f = 2^{j d} \int_{\mathbb R^d} \widetilde h(2^j y) \, f(x-y)\,dy.
\eeno
For any function $f$ in the usual Schwarz class $\mathcal{S}$, (\ref{aa}) implies
\beno
\widehat f (\xi) = \sum_{j\in \mathbb Z}  \psi(2^{-j} \xi)\,\widehat f (\xi) \qquad \mbox{for all $\xi \in \mathbb R^d$}
\eeno
when $f$ satisfies, for any $x^\beta$ in the set of all polynomials $\mathcal P$,
$$
\int_{\mathbb R^d} x^\beta \,f(x)\,dx =0.
$$
In order to write the Littlewood-Paley decomposition for $F\in \mathcal S'$, we need to restrict to the subspace $\mathcal S_h'$ consisting of $f\in \mathcal S'$ satisfying
$$
\lim_{j \to -\infty}  \mathring S_j f =0\quad \mbox{in}\quad \mathcal S'.
$$
Any $f\in \mathcal S'$ that has a locally integrable Fourier transform is in $\mathcal S_h'$.

\vskip .1in
The homogeneous Besov space can be defined in terms of  $\mathring \Delta_j$ specified above.
\begin{define}
	The homogeneous Besov space $\mathring B^s_{p,q}$ with $1\le p,q \le \infty$
	and $s\in {\mathbb R}$ consists of  $f\in {\mathcal S}_h'$
	satisfying
	$$
	\|f\|_{\mathring B^s_{p,q}} \equiv \|2^{js} \|\mathring \Delta_j f\|_{L^p} \|_{l^q}
	<\infty.
	$$
\end{define}
In terms of the homogeneous dyadic blocks, we can also write the standard products
in terms of the paraproducts.

\vskip .1in
The following space-time spaces introduced in \cite{ChLe} have been used in the
previous sections.
\begin{define}
	Let $s\in \mathbb R$ and $1\le p,q, r \le \infty$. Let $T\in (0, \infty]$.
	The space-time space $\widetilde L^r(0, T; B^s_{p,q})$ consists of tempered distributions satisfying
	$$
	\|f\|_{L^r(0, T; B^s_{p,q})} \equiv \|2^{js} \|\|\Delta_j f\|_{L^p}\|_{L^r(0, T)} \|_{l^q} <\infty.
	$$
	$\widetilde L^r(0, T; \mathring B^s_{p,q})$ is similarly defined.
\end{define}

By Minkowski's inequality, the standard space-time space $L^r(0, T; B^s_{p,q})$ is related to $\widetilde L^r(0, T; B^s_{p,q})$ as follows
\beno
&& L^r(0, T; B^s_{p,q}) \subsetneq \widetilde L^r(0, T; B^s_{p,q}) \qquad \mbox{if $r < q$},\\
&& \widetilde L^r(0, T; B^s_{p,q}) \subsetneq L^r(0, T; B^s_{p,q}) \qquad \mbox{if $r > q$},\\
&& L^r(0, T; B^s_{p,q}) = \widetilde L^r(0, T; B^s_{p,q}) \qquad \mbox{if $r = q$}.
\eeno

\vskip .1in
Bernstein's inequality is a useful tool on Fourier localized functions and these inequalities trade derivatives for integrability. The following proposition provides Bernstein type inequalities for fractional derivatives.
\begin{lemma}\label{bern}
Let $\alpha\ge0$. Let $1\le p\le q\le \infty$.
\begin{enumerate}
	\item[1)] If $f$ satisfies
	$$
	\mbox{supp}\, \widehat{f} \subset \{\xi\in \mathbb{R}^d: \,\, |\xi|
	\le K 2^j \},
	$$
	for some integer $j$ and a constant $K>0$, then
	$$
	\|(-\Delta)^\alpha f\|_{L^q(\mathbb{R}^d)} \le C_1\, 2^{2\alpha j +
		j d(\frac{1}{p}-\frac{1}{q})} \|f\|_{L^p(\mathbb{R}^d)}.
	$$
	\item[2)] If $f$ satisfies
	\begin{equation*}
	\mbox{supp}\, \widehat{f} \subset \{\xi\in \mathbb{R}^d: \,\, K_12^j
	\le |\xi| \le K_2 2^j \}
	\end{equation*}
	for some integer $j$ and constants $0<K_1\le K_2$, then
	$$
	C_1\, 2^{2\alpha j} \|f\|_{L^q(\mathbb{R}^d)} \le \|(-\Delta)^\alpha
	f\|_{L^q(\mathbb{R}^d)} \le C_2\, 2^{2\alpha j +
		j d(\frac{1}{p}-\frac{1}{q})} \|f\|_{L^p(\mathbb{R}^d)},
	$$
			where $C_1$ and $C_2$ are constants depending on $\alpha,p$ and $q$
	only.
\end{enumerate}
\end{lemma}

\vskip .1in
We now state and prove bounds for the triple products involving Fourier localized functions. These bounds have been used in the previous sections in the proof of Theorem \ref{main}.

\begin{lemma} \label{para}
Let $j\in \mathbb Z$ be an integer. Let $\Delta_j$ be a dyadic block operator (either inhomogeneous or homogeneous).
\begin{enumerate}
	\item Let $F$ be a divergence-free vector field. Then there exists a constant $C$ independent of $j$ such that
\ben
&& \left|\int_{\mathbb R^d} \Delta_j (F\cdot \na G) \cdot \Delta_j H \,dx \right| \notag\\
&\le&\,C\, \|\Delta_j H\|_{L^2} \, \Big(2^j \,{\sum_{m \le j-1}} 2^{\frac{d}{2} m}\, \|\Delta_m F\|_{L^2} \sum_{|j-k| \le 2}\|\Delta_k G\|_{L^2} \notag\\
&& \quad + \sum_{|j-k| \le 2}\|\Delta_k F\|_{L^2} \,{\sum_{m \le j-1} }2^{(1+\frac{d}{2}) m}\, \|\Delta_m G\|_{L^2} \notag\\
&& \quad + \sum_{k\ge j-1} 2^j\, 2^{\frac{d}{2}k}\,  \|\Delta_k F\|_{L^2} \|\widetilde{\Delta}_k G\|_{L^2}\Big). \label{gg1}
\een
\item Let $F$ be a divergence-free vector field. Then there exists a constant $C$ independent of $j$ such that
\ben
&& \left|\int_{\mathbb R^d} \Delta_j (F\cdot \na G) \cdot \Delta_j G \,dx \right| \notag\\
&\le & \,C\, \|\Delta_j G\|_{L^2} \,\Big({\sum_{m \le j-1}} 2^{(1+\frac{d}{2}) m}\, \|\Delta_m F\|_{L^2}\, \sum_{|j-k| \le 2}\|\Delta_k G\|_{L^2} \notag\\
&& \quad + \sum_{|j-k| \le 2}\|\Delta_k F\|_{L^2} \,\sum_{m \le j-1} 2^{(1+\frac{d}{2}) m}\, \|\Delta_m G\|_{L^2} \notag\\
&& \quad + \sum_{k\ge j-1} 2^j\, 2^{\frac{d}{2}k}\,  \|\Delta_k F\|_{L^2} \|\widetilde{\Delta}_k G\|_{L^2}\Big) \label{gg2}
\een
\item Let $F$ be a divergence-free vector field. Then there exists a constant $C$ independent of $j$ such that
\ben
&& \left|\int_{\mathbb R^d} \Delta_j (F\cdot \na H) \cdot \Delta_j G \,dx + \int_{\mathbb R^d} \Delta_j (F\cdot \na G) \cdot \Delta_j H \,dx\right| \notag\\
& \le &\,C\, \|\Delta_j G\|_{L^2} \,\Big({\sum_{m \le j-1}} 2^{(1+\frac{d}{2}) m}\, \|\Delta_m F\|_{L^2}\, \sum_{|j-k| \le 2}\|\Delta_k H\|_{L^2} \notag\\
&& \quad + \sum_{|j-k| \le 2}\|\Delta_k F\|_{L^2} \,\sum_{m \le k-1} 2^{(1+\frac{d}{2}) m}\, \|\Delta_m H\|_{L^2} \notag\\
&& \quad + \sum_{k\ge j-1} 2^j\, 2^{\frac{d}{2}k}\,  \|\Delta_k F\|_{L^2} \|\widetilde{\Delta}_k H\|_{L^2}\Big)  \notag\\
&& \quad + \,C\, \|\Delta_j H\|_{L^2} \,\Big({\sum_{m \le j-1}} 2^{(1+\frac{d}{2}) m}\, \|\Delta_m F\|_{L^2}\, \sum_{|j-k| \le 2}\|\Delta_k G\|_{L^2} \notag\\
&& \quad + \sum_{|j-k| \le 2}\|\Delta_k F\|_{L^2} \,\sum_{m \le k-1} 2^{(1+\frac{d}{2}) m}\, \|\Delta_m G\|_{L^2} \notag\\
&& \quad + \sum_{k\ge j-1} 2^j\, 2^{\frac{d}{2}k}\,  \|\Delta_k F\|_{L^2} \|\widetilde{\Delta}_k G\|_{L^2}\Big). \label{g83}
\een
\end{enumerate}
\end{lemma}

\begin{proof}
The proof of these inequalities essentially follow from the paraproduct decomposition. By the paraproduct decomposition,
\beno
\Delta_j (F\cdot \na G) &=& \sum_{|j-k| \le 2} \Delta_j (S_{k-1} F\cdot \Delta_k\na G) + \sum_{|j-k|\le 2} \Delta_j ( \Delta_k F\cdot S_{k-1}\na G) \\
&& + \sum_{k\ge j-1} \Delta_j (\Delta_k F
\cdot  \na \widetilde{\Delta}_k G).
\eeno
By H\"{o}lder's inequality and Bernstein's inequality in Lemma \ref{bern},
\beno
\left|\int_{\mathbb R^d} \Delta_j (F\cdot \na G) \cdot \Delta_j H \,dx \right|
&\le& \|\Delta_j H\|_{L^2} \, \Big(\sum_{|j-k| \le 2} 2^k \,\|S_{k-1} F\|_{L^\infty} \, \|\Delta_k G\|_{L^2} \\
&& + \sum_{|j-k| \le 2} \|\Delta_k F\|_{L^2}\, \|S_{k-1}\na G||_{L^\infty} \\
&& + \sum_{k\ge j-1} 2^j \, \|\Delta_k F\|_{L^2} \|\widetilde{\Delta}_k G\|_{L^\infty}
\Big),
\eeno
	where we have used $\na\cdot F=0$ in the last part. (\ref{gg1}) then follows if we invoke the inequalities of the form
\beq\label{gg3}
\|S_{k-1} F\|_{L^\infty} \le \sum_{m\le k-2} 2^{\frac{d}{2} m}\,  \|\Delta_m F\|_{L^2}.
\eeq
To prove (\ref{gg2}), we further write the first term as the sum of a commutator and two correction terms,
\beno
\Delta_j (F\cdot \na G)&=& \sum_{|j-k| \le 2} [\Delta_j,  S_{k-1} F\cdot\na]
\Delta_k G\\
&& + \sum_{|j-k| \le 2} (S_{k-1} F -S_j F)
\cdot \Delta_j\Delta_k \na G\\
&& + \, S_j F \cdot\na \Delta_j G + \sum_{|j-k|\le 2} \Delta_j ( \Delta_k F\cdot S_{k-1}\na G) \\
&& + \sum_{k\ge j-1} \Delta_j (\Delta_k F
\cdot  \na \widetilde{\Delta}_k G).
\eeno
Due to $\na\cdot F=0$,
$$
\int_{\mathbb R^d}  S_j F \cdot\na \Delta_j G\cdot \Delta_j G \,dx =0.
$$
By  H\"{o}lder's inequality, Bernstein's inequality and a commutator estimate,
\beno
\left|\int_{\mathbb R^d} \Delta_j (F\cdot \na G) \cdot \Delta_j G \,dx \right|
&\le& \|\Delta_j G\|_{L^2} \,\Big(\sum_{|j-k|\le 2} \|\na S_{k-1} F\|_{L^\infty} \|\Delta_k G\|_{L^2} \\
&& + \,C\, 2^{(1+\frac{d}{2})j} \,\sum_{|j-k| \le 2} \|\Delta_k F\|_{L^2}\, \|\Delta_j G \|_{L^2} \\
&& + \sum_{|j-k|\le 2} \|\Delta_k F\|_{L^2} \,\|S_{k-1}\na G\|_{L^\infty} \\
&& + \sum_{k\ge j-1} 2^j\, 2^{\frac{d}{2} k}\,  \|\Delta_k F\|_{L^2} \, \|\widetilde{\Delta}_k G \|_{L^2} \Big).
\eeno
(\ref{gg2}) then follows when we invoke similar inequalities as (\ref{gg3}). The proof of (\ref{g83}) is very similar. This completes the proof of Lemma \ref{para}.
\end{proof}

\vskip .3in
\centerline{\bf Acknowledgments}

\vskip .1in
Q. Jiu was partially supported by the National Natural Science Foundation of
China (NNSFC) (No. 11671273) and Beijing Natural Science Foundation (BNSF) (No. 1192001).
J. Wu was partially
supported by the National Science Foundation of the United States under grant
number DMS-1614246 and by the AT\&T Foundation at the Oklahoma State University.

\vskip .3in
\bibliographystyle{plain}

\begin{thebibliography}{10}
	
\bibitem{Alf}  H. Alfv\'{e}n, Existence of electromagnetic-hydrodynamic waves,
{\it Nature \bf 150} (1942),  405-406.

\bibitem{BCD} H. Bahouri, J.-Y. Chemin and R. Danchin, \textit{Fourier Analysis and Nonlinear Partial Differential Equations}, Springer, 2011.

\bibitem{BSS} C. Bardos, C. Sulem and P.L. Sulem, Longtime dynamics of a conductive fluid in the presence of a strong magnetic field, {\it Trans. Am. Math. Soc. \bf 305} (1988), 175--191.


\bibitem{BL} J. Bergh and J. L\"{o}fstr\"{o}m, {\it Interpolation Spaces,
	An Introduction}, Springer-Verlag, Berlin-Heidelberg-New York, 1976.	

\bibitem{Bis} D. Biskamp, {\it Nonlinear Magnetohydrodynamics}, Cambridge University Press, Cambridge, 1993.

\bibitem{CaiLei} Y. Cai and Z. Lei, Global well-posedness of the incompressible magnetohydrodynamics, {\it Arch. Ration. Mech. Anal. \bf 2018} (228), No.3, 969--993.

\bibitem{CaoReWu} C. Cao, D. Regmi and J. Wu, The 2D MHD equations with horizontal dissipation and horizontal
magnetic diffusion, {\it J. Differential Equations \bf 254} (2013), 2661-2681.

\bibitem{CaoWu} C. Cao and J. Wu, Global regularity for the 2D MHD equations with mixed partial dissipation
and magnetic diffusion, {\it Adv.  Math. \bf 226} (2011), 1803-1822.

\bibitem{CaoWuYuan} C. Cao, J. Wu and B. Yuan, The 2D incompressible
magnetohydrodynamics equations with only magnetic diffusion, {\it SIAM J. Math. Anal.} {\bf 46} (2014),  {588-602}.

\bibitem{ChLe} J. Chemin and N. Lerner, Flot de champs de vecteurs non lipschitziens
et \'{e}quations de Navier-Stokes, {\it J. Differential Equations \bf 121} (1995), 314-328.

\bibitem{CMRR} J. Chemin, D. McCormick, J. Robinson and J. Rodrigo, Local existence for the non-resistive MHD equations in Besov spaces, {\it Adv. Math. \bf 286} (2016), 1-31.

\bibitem{ChDai} A. Cheskidov and M. Dai,  Norm inflation for generalized magneto-hydrodynamic system, {\it Nonlinearity \bf 28} (2015),  129-142.


\bibitem{Davi} P.A. Davidson, {\it An Introduction to Magnetohydrodynamics},  Cambridge University Press, Cambridge, England, 2001.

\bibitem{DongLiWu1} B. Dong, Y. Jia, J. Li and J. Wu, Global regularity and time decay for the 2D magnetohydrodynamic equations with fractional dissipation and partial magnetic diffusion, {\it J. Math. Fluid Mechanics \bf 20} (2018), {1541-1565}.

\bibitem{DongLiWu0} B. Dong, J. Li and J. Wu, Global regularity for the 2D MHD equations with partial hyperresistivity, {\it International Math Research Notices}, 2018,  rnx240, https://doi.org/10.1093  /imrn/rnx240.


\bibitem{DuLi} G. Duvaut and J. L. Lions, Inequations en thermoelasticite et magnetohydrodynamique, {\it Arch. Ration. Mech. Anal. \bf 46} (1972),  241-279.

\bibitem{FNZ} J. Fan, H. Malaikah, S. Monaquel, G. Nakamura and  Y. Zhou, Global Cauchy problem of 2D generalized MHD equations, {\it Monatsh. Math. \bf 175} (2014),  127-131.


\bibitem{FMRR1} C. Fefferman, D. McCormick, J. Robinson and J. Rodrigo, Higher order commutator estimates and local existence for the non-resistive MHD equations and related models, {\it J. Funct. Anal. \bf 267} (2014), 1035-1056.

\bibitem{FMRR2} C. Fefferman, D. McCormick, J. Robinson and J. Rodrigo, Local existence for the non-resistive MHD equations in nearly optimal Sobolev spaces, {\it Arch. Ration. Mech. Anal. 223} (2017), 677-691.

\bibitem{HeXuYu}L. He, L. Xu and P. Yu, On global dynamics of three dimensional magnetohydrodynamics: nonlinear stability
of Alfv\'{e}n waves, {\it Ann. PDE \bf 4} (2018), Art.5, 105 pp.

\bibitem{HuX} X. Hu, Global existence for two dimensional compressible
magnetohydrodynamic flows with zero magnetic diffusivity, arXiv: 1405.0274v1 [math.AP] 1 May 2014.

\bibitem{HuLin} X. Hu and F. Lin, Global Existence for Two Dimensional Incompressible Magnetohydrodynamic Flows with Zero Magnetic Diffusivity, arXiv: 1405.0082v1 [math.AP] 1 May 2014.

\bibitem{JiuNiu} Q. Jiu and D. Niu, Mathematical results related to a two-dimensional magneto-hydrodynamic equations
{\it Acta Math. Sci. Ser. B Engl. Ed. \bf 26} (2006), 744-756.

\bibitem{JNW} Q. Jiu, D. Niu, J. Wu, X. Xu and H. Yu, The 2D magnetohydrodynamic equations with magnetic diffusion, {\it Nonlinearity \bf 28} (2015),  3935-3955.

\bibitem{JiuZhao2} Q. Jiu and J. Zhao, Global regularity of 2D generalized MHD equations with magnetic diffusion, {\it Z. Angew. Math. Phys. \bf 66} (2015), 677-687.

\bibitem{LTY} J. Li,  W. Tan,  Z. Yin, Local existence and uniqueness for the non-resistive MHD equations in homogeneous Besov spaces, {\it Adv. Math.} {\bf 317} (2017), 786-798.

\bibitem{LinZhang1} F. Lin, L. Xu, and  P. Zhang, Global small solutions to 2-D incompressible MHD system, {\it J. Differential Equations} {\bf 259} (2015), 5440--5485.



\bibitem{MWZ}C. Miao, J. Wu and Z. Zhang, \textit{Littlewood-Paley Theory and its Applications in Partial Differential
	Equations of Fluid Dynamics}, Science Press, Beijing, China, 2012 (in Chinese).

\bibitem{PanZhouZhu} R. Pan, Y. Zhou and Y. Zhu,  Global classical solutions of three dimensional viscous MHD system without magnetic diffusion on periodic boxes, {\it  Archive for Rational Mechanics and Analysis \bf  227} (2018), No.2, 637--662.

\bibitem{Ren} X. Ren, J. Wu, Z. Xiang and Z. Zhang, Global existence and decay of smooth solution for the 2-D MHD equations without magnetic diffusion, {\it J. Functional Analysis \bf 267} (2014),  503--541.

\bibitem{Ren2} X. Ren, Z. Xiang and Z. Zhang, Global well-posedness for the 2D MHD equations without magnetic diffusion in a strip domain, {\it Nonlinearity \bf 29} (2016), No.4, 1257--1291.


\bibitem{RS} T. Runst and W. Sickel, \textit{Sobolev Spaces of fractional order,
	Nemytskij operators and Nonlinear Partial Differential Equations}, Walter de Gruyter, Berlin, New York, 1996.


\bibitem{SeTe} M. Sermange and R. Temam, Some mathematical questions related to the MHD equations, {\it Comm. Pure Appl. Math. \bf 36} (1983),  635-664.

\bibitem{Tan} Z. Tan and Y. Wang, Global well-posedness of an initial-boundary value problem for viscous non-resistive MHD systems,  {\it SIAM J. Math. Anal. \bf 50} (2018), No.1, 1432--1470.


\bibitem{Tri} H. Triebel, \textit{Theory of Function Spaces II}, Birkhauser Verlag, 1992.

\bibitem{Wan} R. Wan, On the uniqueness for the 2D MHD equations without magnetic diffusion, {\it Nonlinear Anal. Real World Appl. \bf 30} (2016),  32-40.

\bibitem{WeiZ} D. Wei and Z. Zhang, Global well-posedness of the MHD equations in a homogeneous magnetic field, {\it Anal. PDE \bf 10} (2017), No.6,  1361--1406.




\bibitem{Wu2} J. Wu,  Generalized MHD equations, {\it J. Differential Equations \bf 195} (2003),  284-312.

\bibitem{Wu4} J. Wu, Global regularity for a class of generalized magnetohydrodynamic equations, {\it J. Math. Fluid Mech. \bf 13} (2011), 295-305.

\bibitem{WuMHD2018}J. Wu, The 2D magnetohydrodynamic equations with partial or fractional dissipation,  {\it Lectures on the analysis of nonlinear partial differential equations}, Morningside Lectures on Mathematics, Part 5, MLM5, pp. 283-332, International Press, Somerville, MA, 2018.

\bibitem{WuWu} J. Wu and Y. Wu, Global small solutions to the compressible 2D magnetohydrodynamic system without magnetic diffusion, {\it Adv. Math. \bf 310} (2017), 759--888.

\bibitem{WuWuXu} J. Wu, Y. Wu and X. Xu, Global small solution to the 2D MHD system with a velocity damping term, {\it SIAM J. Math. Anal. \bf 47} (2015), 2630-2656.


\bibitem{Yam1} K. Yamazaki, On the global well-posedness of N-dimensional generalized MHD system in anisotropic spaces, {\it Adv. Differential Equations \bf 19} (2014),  201-224.

\bibitem{Yam2} K. Yamazaki, Remarks on the global regularity of the two-dimensional magnetohydrodynamics system with zero dissipation, {\it Nonlinear Anal. \bf 94} (2014), 194-205.

\bibitem{Yam3} K. Yamazaki, On the global regularity of two-dimensional generalized magnetohydrodynamics system, {\it J. Math. Anal. Appl. \bf 416} (2014), 99-111.


\bibitem{Yam4} K. Yamazaki, Global regularity of logarithmically supercritical MHD system with zero diffusivity, {\it Appl. Math. Lett. \bf 29} (2014), 46-51.


\bibitem{WuYang} W. Yang, Q. Jiu and J. Wu, The 3D incompressible magnetohydrodynamic equations with fractional partial dissipation, {\it J. Differential Equations \bf 266} (2019),  {630-652}.

\bibitem{YuanZhao} B. Yuan and J. Zhao, Global regularity of 2D almost resistive MHD equations, {\it Nonlinear Anal. Real World Appl. \bf 41} (2018), {53-65}.



\bibitem{TZhang} T. Zhang, An elementary proof of the global existence and uniqueness theorem to 2-D incompressible non-resistive MHD system, arXiv:1404.5681v1 [math.AP] 23 Apr 2014.


\end{thebibliography}

\end{document}